\documentclass[twoside]{amsart}
\newif\ifdraft
\draftfalse
\usepackage{amsmath,amsfonts,amsthm,mathrsfs}
\usepackage{amssymb}

\usepackage[unicode,bookmarks]{hyperref}
\usepackage[usenames,dvipsnames]{xcolor}
\hypersetup{colorlinks=true,citecolor=NavyBlue,linkcolor=BrickRed,urlcolor=Orange}

\usepackage{expl3}

\ExplSyntaxOn
\prop_new:N \g_cite_map_prop
\tl_new:N \l_citekey_result_tl

\cs_new:Npn \mapcitekey #1#2 {
  \clist_map_inline:nn {#2}
       {  \prop_gput:Nnn  \g_cite_map_prop  {##1} {#1}   }
}

\cs_new:Npn \getcitekey #1 {
   \prop_get:NoN \g_cite_map_prop{#1}  \l_citekey_result_tl
   \quark_if_no_value:NF \l_citekey_result_tl
       {  \tl_set_eq:NN #1  \l_citekey_result_tl  }
}

\cs_new:Npn \showcitekeymaps {\prop_show:N  \g_cite_map_prop }
\ExplSyntaxOff

\usepackage{etoolbox}
\makeatletter
\patchcmd{\@citex}{\if@filesw}{\getcitekey\@citeb \if@filesw}%
    {\typeout{*** SUCCESS ***}}{\typeout{*** FAIL ***}}
\patchcmd{\nocite}{\if@filesw}{\getcitekey\@citeb \if@filesw}%
    {\typeout{*** SUCCESS ***}}{\typeout{*** FAIL ***}}
\makeatother

\usepackage{enumitem}

\newenvironment{aenumerate}{%
	\begin{enumerate}[label=(\alph{*}), ref=(\alph{*})]
}{%
	\end{enumerate}%
}

\usepackage{chngcntr}

\ifdraft
\usepackage[notcite,notref,color]{showkeys}

\definecolor{labelkey}{gray}{0.5}
\fi

\usepackage{tikz}

\usepackage{tikz-cd}
\tikzset{commutative diagrams/arrow style=math font}

\usetikzlibrary{matrix,arrows}
\newlength{\myarrowsize} 

\pgfarrowsdeclare{cmto}{cmto}{
	\pgfsetdash{}{0pt} 
	\pgfsetbeveljoin 
	\pgfsetroundcap 
	\setlength{\myarrowsize}{0.6pt}
	\addtolength{\myarrowsize}{.5\pgflinewidth}
	\pgfarrowsleftextend{-4\myarrowsize-.5\pgflinewidth} 
	\pgfarrowsrightextend{.8\pgflinewidth}
}{
	\setlength{\myarrowsize}{0.6pt} 
  	\addtolength{\myarrowsize}{.5\pgflinewidth}  
	\pgfsetlinewidth{0.5\pgflinewidth}
	\pgfsetroundjoin
	\pgfpathmoveto{\pgfpoint{1.5\pgflinewidth}{0}}
	\pgfpatharc{-109}{-170}{4\myarrowsize}
	\pgfpatharc{10}{189}{0.58\pgflinewidth and 0.2\pgflinewidth}
	\pgfpatharc{-170}{-115}{4\myarrowsize+\pgflinewidth}
	\pgfpathclose
	\pgfusepathqfillstroke
	\pgfpathmoveto{\pgfpoint{1.5\pgflinewidth}{0}}
	\pgfpatharc{109}{170}{4\myarrowsize}
	\pgfpatharc{-10}{-189}{0.58\pgflinewidth and 0.2\pgflinewidth}
	\pgfpatharc{170}{115}{4\myarrowsize+\pgflinewidth}
	\pgfpathclose
	\pgfusepathqfillstroke
	\pgfsetlinewidth{2\pgflinewidth}
}

\pgfarrowsdeclare{cmonto}{cmonto}{
	\pgfsetdash{}{0pt} 
	\pgfsetbeveljoin 
	\pgfsetroundcap 
	\setlength{\myarrowsize}{0.6pt}
	\addtolength{\myarrowsize}{.5\pgflinewidth}
	\pgfarrowsleftextend{-4\myarrowsize-.5\pgflinewidth} 
	\pgfarrowsrightextend{.8\pgflinewidth}
}{
	\setlength{\myarrowsize}{0.6pt} 
  	\addtolength{\myarrowsize}{.5\pgflinewidth}  
	\pgfsetlinewidth{0.5\pgflinewidth}
	\pgfsetroundjoin
	\pgfpathmoveto{\pgfpoint{1.5\pgflinewidth}{0}}
	\pgfpatharc{-109}{-170}{4\myarrowsize}
	\pgfpatharc{10}{189}{0.58\pgflinewidth and 0.2\pgflinewidth}
	\pgfpatharc{-170}{-115}{4\myarrowsize+\pgflinewidth}
	\pgfpathclose
	\pgfusepathqfillstroke
	\pgfpathmoveto{\pgfpoint{1.5\pgflinewidth}{0}}
	\pgfpatharc{109}{170}{4\myarrowsize}
	\pgfpatharc{-10}{-189}{0.58\pgflinewidth and 0.2\pgflinewidth}
	\pgfpatharc{170}{115}{4\myarrowsize+\pgflinewidth}
	\pgfpathclose
	\pgfusepathqfillstroke
	\pgfpathmoveto{\pgfpoint{1.5\pgflinewidth-0.3em}{0}}
	\pgfpatharc{-109}{-170}{4\myarrowsize}
	\pgfpatharc{10}{189}{0.58\pgflinewidth and 0.2\pgflinewidth}
	\pgfpatharc{-170}{-115}{4\myarrowsize+\pgflinewidth}
	\pgfpathclose
	\pgfusepathqfillstroke
	\pgfpathmoveto{\pgfpoint{1.5\pgflinewidth-0.3em}{0}}
	\pgfpatharc{109}{170}{4\myarrowsize}
	\pgfpatharc{-10}{-189}{0.58\pgflinewidth and 0.2\pgflinewidth}
	\pgfpatharc{170}{115}{4\myarrowsize+\pgflinewidth}
	\pgfpathclose
	\pgfusepathqfillstroke
	\pgfsetlinewidth{2\pgflinewidth}
}

\pgfarrowsdeclare{cmhook}{cmhook}{
	\pgfsetdash{}{0pt} 
	\pgfsetbeveljoin 
	\pgfsetroundcap 
	\setlength{\myarrowsize}{0.6pt}
	\addtolength{\myarrowsize}{.5\pgflinewidth}
	\pgfarrowsleftextend{-4\myarrowsize-.5\pgflinewidth} 
	\pgfarrowsrightextend{.8\pgflinewidth}
}{
	\setlength{\myarrowsize}{0.6pt} 
  	\addtolength{\myarrowsize}{.5\pgflinewidth}  
 	\pgfsetdash{}{0pt}
	\pgfsetroundcap
	\pgfpathmoveto{\pgfqpoint{0pt}{-4.667\pgflinewidth}}
	\pgfpathcurveto
    {\pgfqpoint{4\pgflinewidth}{-4.667\pgflinewidth}}
    {\pgfqpoint{4\pgflinewidth}{0pt}}
    {\pgfpointorigin}
	\pgfusepathqstroke
}

\newenvironment{diagram*}[2]{%
\[%
\begin{tikzpicture}[>=cmto,baseline=(current bounding box.center),%
	to/.style={->,font=\scriptsize,cap=round},%
	into/.style={cmhook->,font=\scriptsize,cap=round},%
	onto/.style={-cmonto,font=\scriptsize,cap=round},%
	math/.style={matrix of math nodes, row sep=#2, column sep=#1,%
		text height=1.5ex, text depth=0.25ex}]%
}{%
\end{tikzpicture}%
\]%
\ignorespacesafterend%
}

\newcommand{\Dmod}{\mathscr{D}}
\newcommand{\Mmod}{\mathcal{M}}

\newcommand{\decal}[1]{\lbrack #1 \rbrack}

\newcommand{\shH}{\mathcal{H}}

\newcommand{\norm}[1]{\lVert#1\rVert}

\newcommand{\abs}[1]{\lvert #1 \rvert}

\newcommand{\eps}{\varepsilon}

\newcommand{\tensor}{\otimes}

\newcommand{\NN}{\mathbb{N}}
\newcommand{\ZZ}{\mathbb{Z}}
\newcommand{\QQ}{\mathbb{Q}}
\newcommand{\RR}{\mathbb{R}}
\newcommand{\CC}{\mathbb{C}}
\newcommand{\HH}{\mathbb{H}}
\newcommand{\PP}{\mathbb{P}}

\newcommand{\pder}[2]{\frac{\partial #1}{\partial #2}}

\newcommand{\vfeld}[1]{\frac{\partial}{\partial #1}}

\newcommand{\menge}[2]{\bigl\{ \thinspace #1 \thinspace\thinspace \big\vert%
\thinspace\thinspace #2 \thinspace \bigr\}}

\newcommand{\mengge}[2]{\biggl\{ \thinspace #1 \thinspace\thinspace \bigg\vert%
\thinspace\thinspace #2 \thinspace \biggr\}}

\DeclareMathOperator{\im}{im}

\DeclareMathOperator{\Res}{Res}
\DeclareMathOperator{\Spec}{Spec}

\DeclareMathOperator{\id}{id}

\renewcommand{\Im}{\operatorname{Im}}
\renewcommand{\Re}{\operatorname{Re}}

\DeclareMathOperator{\Sym}{Sym}
\DeclareMathOperator{\gr}{gr}

\DeclareMathOperator{\End}{End}
\DeclareMathOperator{\Hom}{Hom}

\DeclareMathOperator{\SL}{SL}

\newcommand{\define}[1]{\emph{#1}}

\newcommand{\glie}{\mathfrak{g}}

\newcommand{\lie}[2]{\lbrack #1, #2 \rbrack}

\newcommand{\biglie}[2]{\bigl[ #1, #2 \bigr]}

\newcommand{\shf}[1]{\mathscr{#1}}
\newcommand{\OX}{\shf{O}_X}
\newcommand{\OmX}{\Omega_X}

\newcommand{\shHZ}{\shH_{\ZZ}}

\newcommand{\defeq}{\underset{\textrm{def}}{=}}

\newcommand{\restr}[1]{\big\vert_{#1}}

\newcommand{\argbl}{-}

\def\overbar#1#2#3{{%
	\setbox0=\hbox{\displaystyle{#1}}%
	\dimen0=\wd0
	\advance\dimen0 by -#2 
	\vbox {\nointerlineskip \moveright #3 \vbox{\hrule height 0.3pt width \dimen0}%
		\nointerlineskip \vskip 1.5pt \box0}%
}}

\newcommand{\dst}{\Delta^{\ast}}
\newcommand{\dstn}[1]{(\dst)^{#1}}

\newcommand{\into}{\hookrightarrow}

\newcommand{\HR}{H_{\RR}}
\newcommand{\HZ}{H_{\ZZ}}
\newcommand{\HC}{H_{\CC}}
\newcommand{\class}[1]{\lbrack #1 \rbrack}

\newcommand{\jl}{j_{\ast}}

\newcommand{\fu}{f^{\ast}}

\newcommand{\tl}{t_{\ast}}

\newcommand{\shO}{\shf{O}}

\makeatletter
\let\@@seccntformat\@seccntformat
\renewcommand*{\@seccntformat}[1]{%
  \expandafter\ifx\csname @seccntformat@#1\endcsname\relax
    \expandafter\@@seccntformat
  \else
    \expandafter
      \csname @seccntformat@#1\expandafter\endcsname
  \fi
    {#1}%
}
\newcommand*{\@seccntformat@subsection}[1]{%
  \textbf{\csname the#1\endcsname.}
}
\makeatother

\makeatletter
\let\@paragraph\paragraph
\renewcommand*{\paragraph}[1]{%
	\vspace{0.3\baselineskip}%
	\@paragraph{\textit{#1}}%
}
\makeatother

\counterwithin{equation}{subsection}
\counterwithout{subsection}{section}
\counterwithin{figure}{subsection}

\newtheorem{theorem}[equation]{Theorem}
\newtheorem*{theorem*}{Theorem}
\newtheorem{lemma}[equation]{Lemma}
\newtheorem*{lemma*}{Lemma}
\newtheorem{corollary}[equation]{Corollary}
\newtheorem{proposition}[equation]{Proposition}
\newtheorem*{proposition*}{Proposition}

\newtheorem*{conjecture*}{Conjecture}

\theoremstyle{definition}
\newtheorem{definition}[equation]{Definition}
\newtheorem*{definition*}{Definition}
\theoremstyle{remark}

\newtheorem{example}[equation]{Example}
\newtheorem*{example*}{Example}
\newtheorem*{problem*}{Problem}

\theoremstyle{plain}

\newcommand{\theoremref}[1]{\hyperref[#1]{Theorem~\ref*{#1}}}
\newcommand{\lemmaref}[1]{\hyperref[#1]{Lemma~\ref*{#1}}}
\newcommand{\definitionref}[1]{\hyperref[#1]{Definition~\ref*{#1}}}
\newcommand{\propositionref}[1]{\hyperref[#1]{Proposition~\ref*{#1}}}
\newcommand{\conjectureref}[1]{\hyperref[#1]{Conjecture~\ref*{#1}}}
\newcommand{\corollaryref}[1]{\hyperref[#1]{Corollary~\ref*{#1}}}
\newcommand{\exampleref}[1]{\hyperref[#1]{Example~\ref*{#1}}}

\makeatletter
\let\old@caption\caption
\renewcommand*{\caption}[1]{%
	\setcounter{figure}{\value{equation}}%
	\stepcounter{equation}%
	\old@caption{#1}\relax%
}
\makeatother

\newcounter{intro}

\newtheorem{intro-conjecture}[intro]{Conjecture}
\newtheorem{intro-corollary}[intro]{Corollary}
\newtheorem{intro-theorem}[intro]{Theorem}

\newcommand{\parref}[1]{\hyperref[#1]{\S\ref*{#1}}}
\newcommand{\chapref}[1]{\hyperref[#1]{Chapter~\ref*{#1}}}

\makeatletter
\newcommand*\if@single[3]{%
  \setbox0\hbox{${\mathaccent"0362{#1}}^H$}%
  \setbox2\hbox{${\mathaccent"0362{\kern0pt#1}}^H$}%
  \ifdim\ht0=\ht2 #3\else #2\fi
  }
\newcommand*\rel@kern[1]{\kern#1\dimexpr\macc@kerna}
\newcommand*\widebar[1]{\@ifnextchar^{{\wide@bar{#1}{0}}}{\wide@bar{#1}{1}}}
\newcommand*\wide@bar[2]{\if@single{#1}{\wide@bar@{#1}{#2}{1}}{\wide@bar@{#1}{#2}{2}}}
\newcommand*\wide@bar@[3]{%
  \begingroup
  \def\mathaccent##1##2{%
    \if#32 \let\macc@nucleus\first@char \fi
    \setbox\z@\hbox{$\macc@style{\macc@nucleus}_{}$}%
    \setbox\tw@\hbox{$\macc@style{\macc@nucleus}{}_{}$}%
    \dimen@\wd\tw@
    \advance\dimen@-\wd\z@
    \divide\dimen@ 3
    \@tempdima\wd\tw@
    \advance\@tempdima-\scriptspace
    \divide\@tempdima 10
    \advance\dimen@-\@tempdima
    \ifdim\dimen@>\z@ \dimen@0pt\fi
    \rel@kern{0.6}\kern-\dimen@
    \if#31
      \overline{\rel@kern{-0.6}\kern\dimen@\macc@nucleus\rel@kern{0.4}\kern\dimen@}%
      \advance\dimen@0.4\dimexpr\macc@kerna
      \let\final@kern#2%
      \ifdim\dimen@<\z@ \let\final@kern1\fi
      \if\final@kern1 \kern-\dimen@\fi
    \else
      \overline{\rel@kern{-0.6}\kern\dimen@#1}%
    \fi
  }%
  \macc@depth\@ne
  \let\math@bgroup\@empty \let\math@egroup\macc@set@skewchar
  \mathsurround\z@ \frozen@everymath{\mathgroup\macc@group\relax}%
  \macc@set@skewchar\relax
  \let\mathaccentV\macc@nested@a
  \if#31
    \macc@nested@a\relax111{#1}%
  \else
    \def\gobble@till@marker##1\endmarker{}%
    \futurelet\first@char\gobble@till@marker#1\endmarker
    \ifcat\noexpand\first@char A\else
      \def\first@char{}%
    \fi
    \macc@nested@a\relax111{\first@char}%
  \fi
  \endgroup
}
\makeatother

\mapcitekey{Cattani+Deligne+Kaplan}{CDK}
\mapcitekey{Schnell:Neron}{Schnell-N}
\mapcitekey{Schnell:Residues}{Schnell-R}
\mapcitekey{Grauert+Remmert:Raeume}{GR}
\mapcitekey{Grauert+Peternell+Remmert:SCV7}{SCV7}
\mapcitekey{Cattani+Kaplan:Luminy}{CK}
\mapcitekey{Cattani+Kaplan:Algebraicity}{CK-survey}
\mapcitekey{Cattani+Kaplan+Schmid:Degeneration}{CKS}
\mapcitekey{Schmid:VHS}{Schmid}
\mapcitekey{Deligne:EquationsDifferentielles}{Deligne}
\mapcitekey{Kashiwara:AsymptoticBehavior}{Kashiwara}
\mapcitekey{Saito:MixedHodgeModules}{Saito-MHM}
\mapcitekey{Cartan:Quotients}{Cartan}
\mapcitekey{Klemm+Maulik+Pandharipande+Scheidegger:YauZaslow}{KMPS}
\mapcitekey{Voisin:Density}{Voisin}

\newcommand{\TZ}{T_{\ZZ}}
\newcommand{\TZt}{\widebar{T}_{\ZZ}}
\newcommand{\Hdgt}{\widebar{\Hdg}}
\newcommand{\epst}{\widebar{\eps}}
\DeclareMathOperator{\Hdg}{Hdg}
\newcommand{\Np}{N^{+}}
\newcommand{\Fsh}{F_{\sharp}}
\newcommand{\Phinil}{\Phi_{\mathit{nil}}}
\newcommand{\shHe}{\tilde{\shH}}
\newcommand{\shHQ}{\shH_{\QQ}}
\newcommand{\shHd}{\shH^{\ast}}
\newcommand{\HQ}{H_{\QQ}}
\newcommand{\CCst}{\CC^{\ast}}
\newcommand{\Fh}{\hat{F}}

\newcommand{\glieR}{\glie_{\RR}}

\newcommand{\Th}{\hat{T}}
\newcommand{\Dck}{\check{D}}
\newcommand{\epstl}{\tilde{\eps}}

\newcommand{\Wtl}{\tilde{W}}
\newcommand{\ttl}{\tilde{t}}
\newcommand{\ztl}{\tilde{z}}
\newcommand{\Ftl}{\tilde{F}}
\newcommand{\Phitl}{\tilde{\Phi}}
\newcommand{\Ttl}{\tilde{T}}
\newcommand{\Ytl}{\tilde{Y}}
\newcommand{\Ntl}{\tilde{N}}
\newcommand{\htl}{\tilde{h}}
\newcommand{\btl}{\tilde{b}}
\newcommand{\etl}{\tilde{e}}
\newcommand{\HCtl}{\tilde{H}_{\CC}}
\newcommand{\HZtl}{\tilde{H}_{\ZZ}}
\newcommand{\HRtl}{\tilde{H}_{\RR}}
\newcommand{\Gammatl}{\tilde{\Gamma}}
\newcommand{\Qtl}{\tilde{Q}}
\newcommand{\HQtl}{\tilde{H}_{\QQ}}

\newcommand{\famX}{\mathscr{X}}

\begin{document}

\title{The extended locus of Hodge classes}

\author{Christian Schnell}
\address{Department of Mathematics, Stony Brook University,
Stony Brook, NY 11794, USA}
\email{\tt cschnell@math.sunysb.edu}

\begin{abstract}
We introduce an ``extended locus of Hodge classes'' that also takes into account
integral classes that become Hodge classes ``in the limit''.
More precisely, given a polarized variation of integral Hodge structure of weight
zero on a Zariski-open subset of a complex manifold, we construct a canonical
analytic space that parametrizes limits of integral classes; the extended locus of
Hodge classes is an analytic subspace that contains the usual locus of Hodge
classes, but is finite and proper over the base manifold.
The construction uses Saito's theory of mixed Hodge modules and a small
generalization of the main technical result of Cattani, Deligne, and Kaplan. 
We study the properties of the resulting analytic space in the case of the
family of hyperplane sections of an odd-dimensional smooth projective variety.
\end{abstract}

\maketitle

\section{Overview}

\subsection{Hodge loci on Calabi-Yau threefolds}

The purpose of this paper is to describe the construction of the \define{extended
locus of Hodge classes} for polarized variations of $\ZZ$-Hodge structure of weight
zero. Before defining things more precisely, we shall consider a typical example
that shows why this is an interesting problem, and what some of the issues are. 

Let $X$ be a smooth projective Calabi-Yau threefold; this means that $\OmX^3 \simeq
\OX$, and that $H^1(X, \OX) = H^2(X, \OX) = 0$. We fix an embedding of $X$ into
projective space, with $\OX(1)$ the corresponding very ample line bundle, and
consider the family of hyperplane sections of $X$. These are parametrized by the
linear system
\[
	B = \bigl\lvert \OX(1) \bigr\rvert,
\]
and we let $B_0 \subseteq B$ denote the open subset that corresponds to smooth hyperplane
sections. Given a cohomology class $\gamma \in H^2(S, \ZZ)$ on a smooth hyperplane
section $S \subseteq X$, we can use parallel transport along paths in $B_0$ to move
$\gamma$ to other hyperplane sections; this operation is of course purely topological
and does not preserve the Hodge decomposition. The \define{Hodge locus of $\gamma$}
is the set
\[
	\menge{b \in B_0}{\text{$\gamma$ can be transported to a Hodge class on $S_b$}}.
\]
Most of these loci are non-empty: in fact, Voisin \cite{Voisin} has proved that the
union of the Hodge loci of all classes $\gamma \in H^2(S, \ZZ)$ is a dense subset of
$B_0$. Since Hodge classes on surfaces are algebraic, the Hodge locus is an algebraic
subvariety of $B_0$; in basic terms, what we are looking at are curves (or
algebraic one-cycles) on $X$ that lie on hyperplane sections. 

We observe that the expected dimension of the Hodge locus is zero. Indeed, a class
$\gamma \in H^2(S, \ZZ)$ is Hodge exactly when it pairs to zero against every
holomorphic two-form on $S$; because $X$ is a Calabi-Yau threefold, we have
\[
	h^0(S, \Omega_S^2) = h^0 \bigl( X, \OmX^3(S) \bigr) - h^0(X, \OmX^3)
			= h^0 \bigl( X, \OX(1) \bigr) - 1 = \dim B.
\]
The number of conditions is the same as the dimension of the parameter space,
and the Hodge locus of $\gamma$ should therefore have a ``virtual'' number of points;
those numbers are of interest in Donaldson-Thomas theory \cite{KMPS}. But there are
two issues that need to be dealt with:

\begin{enumerate}
\item If the Hodge locus actually has finitely many points, one can of course just
count them. But there may be components of positive dimension, and before one can
use excess intersection theory (or some other method) to assign them a number, one
has to compactify such components.
\item An obvious idea is to take the closure of the Hodge locus inside the projective
space $B$; but this is not the right thing to do because there are interesting limit
phenomena that one cannot see in this way.
\end{enumerate}

\begin{example*}
Here is a typical example. Consider a family of hyperplane sections
$S_t \subseteq X$, parametrized by $t \in \Delta$, with $S_t$ smooth for $t \neq 0$,
and $S_0$ having a single ordinary double point. In this case, $H^2(S_t, \ZZ)$ contains a
vanishing cycle $\gamma_t$, namely the class of an embedded two-sphere with
self-intersection number $\gamma_t^2 = -2$. The vanishing cycle is not a Hodge class
on $S_t$, but becomes one ``in the limit''. On the one hand, one has the limit
mixed Hodge structure, which is pure of weight two in this case; $\gamma_t$ is a
Hodge class in this Hodge structure. On the other hand, one can blow up $S_0$ at the
node; the exceptional divisor $E \simeq \PP^1$ satisfies $\class{E}^2 = -2$, and in a
sense, $\class{E}$ is the limit of the $\gamma_t$ as $t \to 0$.
\end{example*}

\subsection{Statement of the problem}

Abstracting from the example above, we now let $\shH$ be an arbitrary polarized
variation of $\ZZ$-Hodge structure of weight zero, defined on a Zariski-open subset
$X_0$ of a smooth projective variety $X$. The assumption about the weight is of
course just for convenience: if $\shH$ has even weight $2k$, we can always replace it by
the Tate twist $\shH(k)$, which has weight zero.  Let $F^p \shH$ denote the Hodge
bundles, and let $\shHZ$ denote the underlying local system. Although it is not
strictly necessary for what follows, we shall assume that the polarization form $Q
\colon \shHQ \tensor \shHQ \to \QQ(0)$ is defined over $\QQ$.

Let us first recall the definition of the usual locus of Hodge classes. The local
system $\shHZ$ determines a (not necessarily connected) covering space 
\[
	\TZ \to X_0, 
\]
whose sheaf of holomorphic sections is isomorphic to $\shHZ$. The points of $\TZ$ are
pairs $(x,h)$, with $h \in \shH_{\ZZ, x}$ a class in the fiber over the
point $x \in X_0$.

\begin{definition*}
The \define{locus of Hodge classes} of $\shH$ is the set
\[
	\Hdg(\shH) = \menge{(x, h) \in \TZ}{%
		\text{$h \in \shH_{\ZZ, x} \cap F^0 \shH_x$ is a Hodge class}}.
\]
\end{definition*}

We consider the locus of Hodge classes (a subset of $\TZ$) instead of the individual
Hodge loci (subsets of $X_0$) because it is useful to keep track of the Hodge classes
themselves: over any given point, there may be more than one such class, and
all of them may be permuted by the monodromy action. On the face of it, 
$\Hdg(\shH)$ is just an analytic subset of the complex manifold $\TZ$; the following
remarkable theorem by Cattani, Deligne, and Kaplan \cite{CDK} shows that it
is actually a countable union of algebraic varieties.

\begin{theorem}[Cattani, Deligne, Kaplan] \label{thm:CDK}
Every connected component of $\Hdg(\shH)$ is an algebraic variety, finite and proper
over $X_0$.
\end{theorem}

This theorem is one of the best results in Hodge theory. When $\shH$ comes
from a family of smooth projective varieties, the Hodge conjecture predicts that
$\Hdg(\shH)$ should be a countable union of algebraic varieties; the point is that
Cattani, Deligne, and Kaplan were able to prove this \emph{without assuming the conjecture}. 

\theoremref{thm:CDK} shows that every connected component of the locus of Hodge
classes can be extended (more or less uniquely) to a projective algebraic variety
that is finite and proper over $X$. As in the example of Calabi-Yau threefolds, this
suggests that we should look for a natural compactification of $\Hdg(\shH)$ that also
takes into account those integral classes that only become Hodge classes ``in the
limit''. The purpose of this paper is to solve that problem with the help of
Saito's theory of mixed Hodge modules. The idea is to construct a complex analytic
space $\TZt$ that extends $\TZ$, and to use it for defining the \define{extended
locus of Hodge classes}. As far as I know, it was Clemens who first suggested working
directly with limits of integral classes; in any case, I learned this idea from
him.

\subsection{The case of a Hodge structure}
\label{par:point}

To motivate the construction, let us first look at the case of a single Hodge
structure $H$. We assume that $H$ is polarized and integral of weight zero; we denote the
underlying $\ZZ$-module by $\HZ$; the polarization by $Q$; and the Hodge filtration
by $F^{\bullet} H$. Let $\Hdg(H) = \HZ \cap F^0 H$ be the set of Hodge classes in
$H$.  According to the bilinear relations, a class $h \in \HZ$ is Hodge exactly when
it is perpendicular (under $Q$) to the space $F^1 H$; this says that $\Hdg(H)$ is
precisely the kernel of the linear mapping
\[
	\eps \colon \HZ \to (F^1 H)^{\ast}, \quad h \mapsto Q(h, \argbl).
\]
At first, it may seem that $\eps$ is not good for much else, because its image is not
a nice subset of $(F^1 H)^{\ast}$. In fact, the dimension of the vector space $F^1 H$
can be much smaller than the rank of $\HZ$, and so $\eps$ will typically have dense
image. But it turns out that the restriction of $\eps$ to the subset
\[
	\HZ(K) = \menge{h \in \HZ}{\abs{Q(h,h)} \leq K}
\]
is well-behaved. The idea of bounding the self-intersection number of the integral
classes already occurs in the paper by Cattani, Deligne, and Kaplan. To back up this
claim, we have the following lemma; note that the estimate in the proof will play an
important role in our analysis later on.

\begin{lemma} \label{lem:fiber}
The mapping $\eps \colon \HZ(K) \to (F^1 H)^{\ast}$ is finite and proper, and its
image is a discrete subset of the vector space $(F^1 H)^{\ast}$.
\end{lemma}

\begin{proof}
We have to show that the preimage of any bounded subset of $(F^1 H)^{\ast}$ is
finite. It will be convenient to measure things in the \emph{Hodge norm}: if
\[
	h = \sum_p h^{p,-p}, \quad \text{with $h^{p,-p} 
		\in F^p H \cap \widebar{F^{-p} H}$,}
\]
is the Hodge decomposition of a vector $h \in H$, then its Hodge norm is
\[
	\norm{h}_H^2 = \sum_p \norm{h^{p,-p}}_H^2  
		= \sum_p (-1)^p Q \Bigl( h^{p,-p}, \widebar{h^{p,-p}} \Bigr).
\]
Now suppose that $h \in \HZ$ satisfies $\abs{Q(h,h)} \leq K$ and $\norm{\eps(h)}_H \leq
R$; it will be enough to prove that $\norm{h}_H$ is bounded by a quantity depending only
on $K$ and $R$. The assumption on $\eps(h)$ means that $\abs{Q(h,v)} \leq R
\norm{v}_H$ for every $v \in F^1 H$. If we apply this inequality to the vector
\[
	v = \sum_{p \geq 1} (-1)^p h^{p,-p},
\]
we find that $\norm{v}_H^2 = \abs{Q(h, v)} \leq R \norm{v}_H$, and hence that
\[
	\sum_{p \geq 1} \norm{h^{p,-p}}_H^2 = \norm{v}_H^2 \leq R^2.
\]
Because $h$ is invariant under conjugation, it follows that $\norm{h}_H^2 \leq
\norm{h^{0,0}}_H^2 + 2 R^2$. This leads to the conclusion that $\norm{h}_H^2 \leq K + 4
R^2$, because
\[
	Q(h,h) = \norm{h^{0,0}}_H^2 + \sum_{p \neq 0} (-1)^p \norm{h^{p,-p}}_H^2 \leq K.
\]
In particular, there are only finitely many possibilities for $h \in \HZ$, which
means that $\eps$ is a finite mapping, and that the image of $\eps$ is a discrete
subset of $(F^1 H)^{\ast}$.
\end{proof}

\subsection{The general case}

Now let us return to the general case. As in \cite{CDK}, it is not actually necessary
to assume that $X$ is projective; we shall therefore consider a polarized variation
of $\ZZ$-Hodge structure $\shH$ of weight zero, defined on a Zariski-open subset
$X_0$ of an arbitrary complex manifold $X$. By performing the construction in
\parref{par:point} at every point of $X_0$, we obtain a holomorphic mapping
\[
	\eps \colon \TZ \to T(F^1 \shH);
\]
here $T(F^1 \shH) = \Spec \bigl( \Sym F^1 \shH \bigr)$ is the holomorphic vector
bundle on $X_0$ whose sheaf of holomorphic sections is $(F^1 \shH)^{\ast}$.
The locus of Hodge classes $\Hdg(\shH)$ is then exactly the preimage of the zero
section in $T(F^1 \shH)$. For any rational number $K \geq 0$, we consider the
submanifold 
\[
	\TZ(K) = \menge{(x,h) \in \TZ}{\abs{Q_x(h,h)} \leq K}.
\]
It is a union of connected components of the covering space $\TZ$, because the
quantity $Q_x(h,h)$ is obviously constant on each connected component. More or less
directly by \lemmaref{lem:fiber}, the holomorphic mapping
\[
	\eps \colon \TZ(K) \to T(F^1 \shH)
\]
is finite and proper, with complex-analytic image; moreover, one can show that the
mapping from $\TZ(K)$ to the normalization of the image is a finite covering space.
For the details, please consult \parref{par:setup} below.

To construct an extension of $\TZ(K)$ to an analytic space over $X$, we use the
theory of Hodge modules \cite{Saito-MHM}. Let $M$ be the polarized Hodge module of
weight $\dim X$ with strict support $X$, canonically associated with $\shH$. We
denote the underlying filtered left $\Dmod$-module by the symbol $(\Mmod, F_{\bullet}
\Mmod)$. The point is that 
\[
	\Mmod \restr{X_0} \simeq \shH \quad \text{and} \quad
		F_k \Mmod \restr{X_0} \simeq F^{-k} \shH;
\]
in particular, the coherent sheaf $F_{-1} \Mmod$ is an extension of the Hodge
bundle $F^1 \shH$ to a coherent sheaf of $\OX$-modules. Now consider the holomorphic
mapping
\[
	\eps \colon \TZ(K) \to T(F_{-1} \Mmod),
\]
where the analytic space on the right-hand side is defined as before as the spectrum
of the symmetric algebra of the coherent sheaf $F_{-1} \Mmod$. We have already seen
that $\eps \bigl( \TZ(K) \bigr)$ is an analytic subset of $T(F^1 \shH)$; since we are
interested in limits of integral classes, we shall extend it to the larger space
$T(F_{-1} \Mmod)$ by taking the closure. The main result of the paper is that the
closure remains analytic.

\begin{theorem} \label{thm:main}
The closure of $\eps \bigl( \TZ(K) \bigr)$ is an analytic subset of $T(F_{-1}
\Mmod)$.
\end{theorem}

The proof consists of two steps: (1) We reduce the problem to the special case where $X
\setminus X_0$ is a divisor with normal crossings and $\shHZ$ has unipotent local
mono\-dromy; this reduction is similar to \cite{Schnell-N}. (2) In that case, we prove
the theorem by a careful local analysis, using the theory of degenerating variations
of Hodge structure. In fact, we deduce the theorem from a strengthening of the main
technical result of Cattani, Deligne, and Kaplan, which we prove by adapting the
method introduced in \cite{CDK}. Rather than just indicating the necessary changes in
their argument, I have chosen to write out a complete proof; I hope that this will
make \chapref{chap:local} useful also to those readers who are only interested in the
locus of Hodge classes and the theorem of Cattani, Deligne, and Kaplan.

Once \theoremref{thm:main} is proved, it makes sense to consider the normalization of
the closure of $\eps \bigl( \TZ(K) \bigr)$. The mapping from $\TZ(K)$ to its
image in the normalization is a finite covering space; it can therefore be extended in a
canonical way to a finite branched covering by appealing to the
\emph{Fortsetzungssatz} of Grauert and Remmert. 

\begin{theorem} 
There is a normal analytic space $\TZt(K)$ containing the complex manifold $\TZ(K)$
as a dense open subset, and a finite holomorphic mapping 
\[
	\epst \colon \TZt(K) \to T(F_{-1} \Mmod),
\]
whose restriction to $\TZ(K)$ agrees with $\eps$. Moreover, $\TZt(K)$ and $\epst$ are
unique up to isomorphism.
\end{theorem}

Since each $\TZ(K)$ is a union of connected components of the covering space $\TZ$,
we can take the union over all the $\TZt(K)$; this operation is well-defined because
of the uniqueness statement in the theorem. In this way, we get a normal analytic space
$\TZt$, and a holomorphic mapping
\[
	\epst \colon \TZt \to T(F_{-1} \Mmod)
\]
with discrete fibers that extends $\eps$. Now the preimage of the zero section in
$T(F_{-1} \Mmod)$ gives us the desired compactification for the locus of Hodge
classes.

\begin{definition} \label{def:Hdg-extended}
The \define{extended locus of Hodge classes} $\Hdgt(\shH)$ is the closed analytic
subscheme $\epst^{-1}(0) \subseteq \TZt$; by construction, it contains the locus of
Hodge classes.
\end{definition}

\subsection{The family of hyperplane sections}

The construction above can be applied to the family of hyperplane sections of a
smooth projective variety of odd dimension. In this case, one has a good description of
the filtered $\Dmod$-module $(\Mmod, F_{\bullet})$ in terms of residues
\cite{Schnell-R}, and it is possible to say more about the space $\TZt$. The fact
that $F_{-1} \Mmod$ is the quotient of an ample vector bundle leads to the following
result; it was predicted by Clemens several years ago.

\begin{theorem}
The analytic space $\TZt(K)$ is holomorphically convex. Every compact analytic
subset of dimension $\geq 1$ lies inside the extended locus of Hodge classes.  
\end{theorem}

\subsection{Acknowledgements}

In writing this paper, I have benefited a lot from a new survey article by Cattani
and Kaplan \cite{CK-survey} that explains the results of \cite{CDK} in the case $n
\leq 2$. I thank Eduardo Cattani for letting me read a draft version, and for
answering some questions. Several years ago, Davesh Maulik asked me about the case of
hyperplane sections of a Calabi-Yau threefold; I thank him for many useful
conversations, and for his general interest in the problem.  Most of all, I thank my
former thesis adviser, Herb Clemens, for suggesting that one should study limits of
integral classes with the help of residues; as in many other cases, his idea
contained the seed for the solution of the general problem.

\section{Local analysis in the normal crossing case}
\label{chap:local}

\subsection{Main result}

The purpose of this chapter is to prove the following special case of
\theoremref{thm:main}. We shall see later how the general case can be reduced to
this one.

\begin{theorem} \label{thm:main-local}
Let $X$ be a complex manifold, and let $X_0 = X \setminus D$ be the complement of a
divisor $D \subseteq X$ with normal crossing singularities. Let $\shH$ be a polarized
variation of $\ZZ$-Hodge structure of weight zero on $X_0$ whose local monodromy at
each point of $D$ is unipotent. Then the closure of the image of the holomorphic
mapping 
\[
	\eps \colon \TZ(K) \to T(F_{-1} \Mmod)
\]
is a complex-analytic subspace of $T(F_{-1} \Mmod)$.
\end{theorem}

Here is a brief outline of the proof. The assertion is local on $X$, and
unaffected by enlarging the divisor $D$. We may therefore assume that $X = \Delta^n$,
with coordinates $s_1, \dotsc, s_n$, and that $D$ is the divisor defined by $s_1
\dotsm s_n = 0$. Denote by $\HZ$ the generic fiber of the local system $\shHZ$, by $Q
\colon \HQ \tensor \HQ \to \QQ$ the pairing on $\HQ = \HZ \tensor_{\ZZ} \QQ$ giving the
polarization, and by $N_1, \dotsc, N_n$ the logarithms of the unipotent monodromy
transformations. Define 
\[
	\HZ(K) = \menge{h \in \HZ}{\abs{Q(h,h)} \leq K}.
\]
After pulling back to the universal covering space $\HH^n$, with coordinates $z_1,
\dotsc, z_n$ (related by $s_j = e^{2 \pi i z_j}$ to the coordinates on $\Delta^n$),
we get a holomorphic mapping 
\[
	\epstl \colon \HH^n \times \HZ(K) \to T(F_{-1} \Mmod).
\]
The main point is to show the following: Suppose that $\bigl( z(m), h(m) \bigr) \in \HH^n
\times \HZ(K)$ is a sequence such that $s_1(m), \dotsc, s_n(m)$ are going to
zero along a bounded sector, and $\epstl \bigl( z(m), h(m) \bigr)$ remains bounded.
Then after passing to a subsequence, $h(m)$ is constant and partially monodromy
invariant. Roughly speaking, this means that if we let $Z \subseteq \Delta^n$ denote
the smallest analytic subvariety containing all the points $\bigl( s_1(m), \dotsc,
s_n(m) \bigr)$, then $h(m)$ is monodromy invariant on $Z$. As explained in
\theoremref{thm:local1} below, such a result quickly leads to a proof
of \theoremref{thm:main-local}.

For technical reasons, we prove a slightly more general result. Let $\Phi
\colon \HH^n \to \HC$ denote the period mapping associated with the variation of
Hodge structure $\shH$ on $\dstn{n}$. Since $\epstl(z, h) = 0$ if and only if $h \in
\Phi^0(z)$, it is reasonable to expect that a bound on $\epstl(z, h)$ should control
the distance between $h$ and $\Phi^0(z)$. In fact, we show that if the sequence
$\epstl \bigl( z(m), h(m) \bigr)$ remains bounded in $T(F_{-1} \Mmod)$, then
\[
	h(m) \equiv b(m) \mod \Phi^0 \bigl( z(m) \bigr).
\]
Here $b(m) \in \HC$ is a bounded sequence with property that, for some $\alpha > 0$,
every $N_j b(m)$ is in $O(e^{-\alpha \Im z_j(m)})$. This uses the description of the
$\Dmod$-module $\Mmod$ in terms of Deligne's canonical extension $\shHe$, and the
fact that $F_{-1} \Mmod$, compared to $F^1 \shHe$, contains
additional sections with poles along the divisors $s_j = 0$. In
\theoremref{thm:local2} below, we prove that even under this weaker assumption,
a subsequence of $h(m)$ is constant and partially monodromy invariant.

In the special case when $b(m)$ is in $O(e^{-\alpha \max_j \Im z_j(m)})$, this result
is due to Cattani, Deligne, and Kaplan \cite[Theorem~2.16]{CDK}. Their proof is an
application of the theory of degenerating variations of Hodge structure, especially
the multi-variable $\SL(2)$-orbit theorem \cite{CKS}. We prove
\theoremref{thm:local2} by adapting their method; there are several difficulties,
caused by the fact that the sequence $b(m)$ is not necessarily going to zero, but
these difficulties can be overcome. As in the original, we argue by induction on the
dimension of $Z$; the description of period mappings in \cite{CKS} lends itself very
well to such an approach.

A subtle point is that the assumption $h(m) \in \HZ$ is needed in many places: it
ensures that certain terms that would only be going to zero when $h(m) \in \HR$ are
actually equal to zero after passing to a subsequence. Rather than giving an abstract
description of the proof, I have decided to include (in \parref{par:one-dim}) a
careful discussion of the special case $n = 1$. All the interesting features of the
general case are present here, but without the added complications of having several
nilpotent operators $N_1, \dotsc, N_n$ and several variables $z_1(m), \dotsc,
z_n(m)$.  Hopefully, this will help the reader understand the proof of the general
case.

\subsection{Local description of the problem}

Since \theoremref{thm:main-local} is evidently a local statement, we shall begin by
reviewing the local description of polarized variations of Hodge structure
\cite{Schmid, Kashiwara, CKS}.  Fortunately, Cattani and Kaplan have written a
beautiful survey article, where they describe all the major results \cite{CK}. Rather
than citing the original sources, I will only quote from this article. 

Let $\Delta^n$, with coordinates $s = (s_1, \dotsc, s_n)$, be the product of $n$
copies of the unit disk; then $\dstn{n}$ is the complement of the divisor defined by
$s_1 \dotsm s_n = 0$. Let $\shH$ be a polarized variation of $\ZZ$-Hodge structure of
weight zero on $\dstn{n}$; we assume that the underlying local system of free
$\ZZ$-modules $\shHZ$ has unipotent monodromy around each of the divisors $s_j = 0$.
Let $\HH^n$, with coordinates $z = (z_1, \dotsc, z_n)$, be the product of $n$ copies
of the upper half-plane; the holomorphic mapping
\[
	\HH^n \to \dstn{n}, \quad 
		z \mapsto \bigl( e^{2 \pi i z_1}, \dotsc, e^{2 \pi i z_n} \bigr)
\]
makes it into the universal covering space of $\dstn{n}$. If we pull back the local
system $\shHZ$, it becomes trivial; let $\HZ$ denote the free $\ZZ$-module of its
global sections, and $Q \colon \HQ \tensor \HQ \to \QQ(0)$ the symmetric bilinear form
coming from the polarization on $\shH$. By assumption, the monodromy transformation
around $s_j = 0$ is of the form $e^{N_j}$, where $N_j$ is a nilpotent endomorphism of
$\HQ = \HZ \tensor_{\ZZ} \QQ$ that satisfies $Q(N_j h_1, h_2) + Q(h_1, N_j h_2) = 0$.
It is clear that $N_1, \dotsc, N_n$ commute.

We now review the description of $\shH$ that results from the work of Cattani,
Kaplan, and Schmid. Let $\Dck$ denote the parameter space for filtrations
$F = F^{\bullet} \HC$ that satisfy $Q(F^p, \widebar{F^q}) = 0$ whenever $p + q > 0$;
let $D \subseteq \Dck$ denote the subset of those $F$ that define a polarized Hodge
structure on $\HC = \HZ \tensor_{\ZZ} \CC$ with polarization $Q$. Recall that $\Dck$
is a closed subvariety of a flag variety, and that the so-called \define{period
domain} $D$ is an open subset of $\Dck$.

The variation of Hodge structure $\shH$ can be lifted to a period mapping
\[
	\Phi \colon \HH^n \to D
\]
which is holomorphic and horizontal. It is known that every element of the cone
\[
	C(N_1, \dotsc, N_n) = \menge{a_1 N_1 + \dotsb + a_n N_n}{a_1, \dotsc, a_n > 0}
\]
defines the same monodromy weight filtration \cite[Theorem~2.3]{CK}; we denote the
common filtration by $W = W(N_1, \dotsc, N_n)$. In the limit, $\shH$ determines
another filtration $F \in \Dck$ for which the pair $(W, F)$ is a mixed Hodge
structure on $\HC$, polarized by $Q$ and every element of $C(N_1, \dotsc, N_n)$.
According to the nilpotent orbit theorem \cite[Theorem~2.1]{CK}, the period mapping
is approximated (with good bounds on the degree of approximation) by the associated
\define{nilpotent orbit}
\[
	\Phinil \colon \HH^n \to D, \quad \Phinil(z) = e^{\sum z_j N_j} F.
\]
One can use the mixed Hodge structure $(W, F)$ to express $\Phi(z)$ in terms of the
nilpotent orbit and additional holomorphic data on $\Delta^n$. Denote by
\[
	\glie = \menge{X \in \End(\HC)}{Q(Xh_1, h_2) + Q(h_1, Xh_2) = 0}
\]
the Lie algebra of infinitesimal isometries of $Q$. The mixed Hodge structure $(W,
F)$ determines a decomposition of $\HC$ with the following properties:
\[
	\HC = \bigoplus_{p,q} I^{p,q}, \quad
	W_w = \bigoplus_{p + q \leq w} I^{p,q}, \quad
	F^k = \bigoplus_{p \geq k} I^{p,q},
\]
A formula for the subspaces $I^{p,q}$ can be found in \cite[(1.12)]{CK}. The
decomposition leads to a corresponding decomposition of the Lie algebra
\[
	\glie = \bigoplus_{p,q} \glie^{p,q},
\]
with $\glie^{p,q}$ consisting of those $X$ that satisfy $X(I^{a,b}) \subseteq
I^{a+p,b+q}$. In this notation, we have $N_1, \dotsc, N_n \in \glie^{-1,-1}$;
moreover, the restriction of $Q$ to the subspace $I^{p,q} \tensor I^{p',q'}$ is
nondegenerate if $p' = -p$ and $q' = -q$, and zero otherwise.

The more precise version of the nilpotent orbit theorem \cite[Theorem~2.8]{CK}
is that the period mapping of $\shH$ can be put into the normal form
\begin{equation} \label{eq:period-std}
	\Phi \colon \HH^n \to D, \quad 
		\Phi(z) = e^{\sum z_j N_j} e^{\Gamma(s)} F,
\end{equation}
for a unique holomorphic mapping
\[
	\Gamma \colon \Delta^n \to \bigoplus_{p \leq -1} \glie^{p,q}
\]
with $\Gamma(0) = 0$. When we write $\Gamma(s)$, it is of course understood that $s_j
= e^{2 \pi i z_j}$ for every $j = 1, \dotsc, n$. The horizontality of the period
mapping has the following very useful consequence \cite[Proposition~2.6]{CK}.

\begin{proposition} \label{prop:Gamma}
Let $\Phi(z) = e^{\sum z_j N_j} e^{\Gamma(s)} F$ be the normal form of a period
mapping on $\HH^n$. Then for every $j = 1, \dotsc, n$, the commutator
\[
	\biglie{N_j}{e^{\Gamma(s)}} = N_j e^{\Gamma(s)} - e^{\Gamma(s)} N_j
\]
vanishes along the divisor $s_j = 0$.
\end{proposition}

The presentation of the period mapping in \eqref{eq:period-std} is also convenient
for describing the polarizable Hodge module $M$ that we obtain by taking the
intermediate extension of $\shH$ to $\Delta^n$. Here is a brief explanation of how
this works.

Let $(\shH, \nabla)$ be the flat vector bundle on $\dstn{n}$ underlying the
variation of Hodge structure. The monodromy being unipotent, this bundle admits a
canonical extension to a vector bundle $\shHe$ on $\Delta^n$, on which the connection
has a logarithmic pole along each of the divisors $s_j = 0$ with nilpotent residue
\cite[Proposition~5.2]{Deligne}. Explicitly, for each $v \in \HC$, the
holomorphic mapping
\begin{equation} \label{eq:canext}
	\sigma_v \colon \HH^n \to \HC, \qquad 
		\sigma_v(z) = e^{\sum z_j N_j} e^{\Gamma(s)} v
\end{equation}
descends to a holomorphic section of $\shH$ on $\dstn{n}$, and $\shHe$ is the
locally free subsheaf of $\jl \shH$ generated by all such sections \cite[(2.2)]{CK}.
The Hodge bundles $F^p \shH$ extend uniquely to holomorphic subbundles $F^p \shHe$ of
the canonical extension; concretely, $F^p \shHe$ is generated by those sections in
\eqref{eq:canext} with $v \in F^p$. Now let $(\Mmod, F)$ denote the
filtered $\Dmod$-module underlying $M$. Then $\Mmod$ is simply the $\Dmod$-submodule
of $\jl \shH$ generated by $\shHe$. Moreover, the Hodge filtration on $\Mmod$ is given by 
\[
	F_k \Mmod = \sum_{j \geq 0} F_j \Dmod_{\Delta^n} \cdot F^{j-k} \shHe.
\]
It satisfies $F_j \Dmod_{\Delta^n} \cdot F_k \Mmod \subseteq F_{j+k} \Mmod$, and each
$F_k \Mmod$ is a coherent sheaf on $\Delta^n$ whose restriction to $\dstn{n}$ agrees
with $F^{-k} \shH$. This is a translation of Saito's results in
\cite[\S3.10]{Saito-MHM}; note that Saito uses right $\Dmod$-modules.
For the purposes of our construction, the important point is that $F_{-1} \Mmod$ has
more sections than $F^1 \shHe$; the following lemma exhibits the ones that we will use.

\begin{lemma} \label{lem:sections-Fp}
For any vector $v \in F^2$, and any index $1 \leq k \leq n$, the formula
\[
	\sigma_{v,k}(z) = e^{\sum z_j N_j} e^{\Gamma(s)} \frac{N_k v}{s_k}
\]
defines a holomorphic section of the coherent sheaf $F_{-1} \Mmod$ on $\Delta^n$.
\end{lemma}

\begin{proof}
It is clear from the description above that $\sigma_v(z)$ is a holomorphic section of
$F_{-2} \Mmod$. By \cite[(2.7)]{CK}, the horizontality of the period mapping is
equivalent to
\[
	d \Bigl( e^{\sum z_j N_j} e^{\Gamma(s)} \Bigr) =
		e^{\sum z_j N_j} e^{\Gamma(s)} 
			\biggl( d \, \Gamma_{-1}(s) + \sum_{j=1}^n N_j dz_j \biggr).
\]
Using this identity and the fact that $s_k = e^{2 \pi i z_k}$, we compute that
\begin{align*}
	\vfeld{s_k} \sigma_v(z) 
		&= e^{\sum z_j N_j} e^{\Gamma(s)} 
			\left( \pder{\Gamma_{-1}(s)}{s_k} + \frac{N_k}{2 \pi i s_k} \right) v \\
		&= e^{\sum z_j N_j} e^{\Gamma(s)} \pder{\Gamma_{-1}(s)}{s_k} v 
			+ \frac{1}{2 \pi i} \sigma_{v,k}(z).
\end{align*}
This section belongs to $F_{-1} \Mmod$ by the definition of the filtration; we now
obtain the result by noting that $\Gamma_{-1}(s) \cdot v$ is a holomorphic
mapping from $\Delta^n$ into $F^1$.
\end{proof}

We close this section by describing the mapping $\eps \colon \TZ \to T(F_{-1}
\Mmod)$ in coordinates. With the conventions in \cite[(1.8)]{CK}, the \'etal\'e space
$\TZ$ of the local system $\shHZ$ can be obtained as the quotient of $\HH^n \times
\HZ$ by the following $\ZZ^n$-action:
\begin{equation} \label{eq:action-TZ}
	a \cdot (z, h) = \bigl( z + a, e^{\sum a_j N_j} h \bigr)
\end{equation}
As in the general problem, we define, for any integer $K \geq 0$, a set
\[
	\HZ(K) = \menge{h \in \HZ}{\abs{Q(h,h)} \leq K}.
\]
Then $\TZ(K)$ is the quotient of $\HH^n \times \HZ(K)$ by the action in
\eqref{eq:action-TZ}. Now $F^1 \shHe$ is a subsheaf of $F_{-1} \Mmod$, and so we have
a commutative diagram
\begin{equation} \label{eq:mappings}
\begin{tikzcd}
\HH^n \times \HZ(K) \rar \arrow[bend left=20]{rr}{\epstl} & \TZ(K) \rar{\eps} \arrow{dr} & 
		T(F_{-1} \Mmod) \dar{p} \\
& & T(F^1 \shHe)
\end{tikzcd}
\end{equation}
The concrete description of $F^1 \shHe$ shows that $T(F^1 \shHe) \simeq \Delta^n \times
\Hom(F^1, \CC)$. We therefore obtain a holomorphic mapping
\[
	\HH^n \times \HZ \to \Delta^n \times \Hom(F^1, \CC),
\]
which, in coordinates, is given by the formula
\[
	(z, h) \mapsto \Bigl( s, 
		v \mapsto Q \bigl( h, e^{\sum z_j N_j} e^{\Gamma(s)} v \bigr) \Bigr).
\]
As usual, the relation $s_j = e^{2 \pi i z_j}$ is implicit in the notation.

\subsection{Reformulation of the problem}

We continue to use the notation introduced in the previous section. Our goal is to
deduce \theoremref{thm:main-local} from the following more precise local statement.

\begin{theorem} \label{thm:local1}
Suppose we are given a sequence of points 
\[
	\bigl( z(m), h(m) \bigr) \in \HH^n \times \HZ(K)
\]
with $x_j(m) = \Re z_j(m)$ bounded and $y_j(m) = \Im z_j(m)$ going to infinity for 
every $j = 1, \dotsc, n$. If $\epstl \bigl( z(m), h(m) \bigr)$ remains bounded inside
$T(F_{-1} \Mmod)$, then there is a subsequence with the following properties: 
\begin{aenumerate}
\item \label{en:local1-a}
The sequence $h(m)$ is constant and equal to $h \in \HZ(K)$.
\item One has $(a_1 N_1 + \dotsb + a_n N_n) h = 0$ for certain positive integers
$a_1, \dotsc, a_n$; in particular, $h \in W_0$.  
\item There is a vector $w \in \CC^n$ such that
\[
	\lim_{m \to \infty} e^{-\sum z_j(m) N_j} h(m) = e^{-\sum w_j N_j} h.
\]
\item Each $N_k h$ is a rational Hodge class in the mixed Hodge structure
\[
	\bigl( W_{-2}, e^{\sum w_j N_j} F \bigr)(-1).
\]
\end{aenumerate}
\end{theorem}

Let us prove that \theoremref{thm:local1} implies \theoremref{thm:main-local}. 
It suffices to show that the closure of the image of $\eps \colon \TZ(K) \to T(F_{-1}
\Mmod)$ is analytic in a neighborhood of any given point in $T(F_{-1} \Mmod)$. After
choosing local coordinates, we may therefore assume without loss of generality that
$X = \Delta^n$ and $X_0 = \dstn{n}$, and consider the behavior of the closure over
the origin. For every $h \in \HZ$, we have a holomorphic mapping
\[
	\epstl_h = \epstl(\argbl, h) \colon \HH^n \to T(F_{-1} \Mmod).
\]
We only need to prove that the closure of the image of $\epstl_h$ is analytic; this is
because, by assertion \ref{en:local1-a} in \theoremref{thm:local1}, any bounded
subset of $T(F_{-1} \Mmod)$ can intersect only finitely many of the sets
$\epstl_h(\HH^n)$. Furthermore, we may assume
that there is a vector $w \in \CC^n$ such that each $N_j h$ is a rational Hodge class
in the mixed Hodge structure $\bigl( W, e^{\sum w_j N_j} F \bigr)(-1)$, and that $a_1
N_1 h + \dotsb + a_n N_n h = 0$ for certain positive integers $a_1, \dotsc, a_n \in
\ZZ$; otherwise, the closure of $\epstl_h(\HH^n)$ does not actually contain any points over
the origin, according to \theoremref{thm:local1}. 

Under these assumptions on $h$, we can prove the stronger result that the image of $p
\circ \epstl_h$ has an analytic closure; here $p \colon T(F_{-1} \Mmod) \to T(F^1
\shHe)$ denotes the holomorphic mapping induced by $F^1 \shHe \into
F_{-1} \Mmod$; see \eqref{eq:mappings}. This suffices to conclude the proof, because
the image of $\epstl_h$ is then contained in the closed analytic subset
\[
	p^{-1} \Bigl( \overline{(p \circ \epstl_h)(\HH^n)} \Bigr).
\]
As $p$ is an isomorphism over $\dstn{n}$, it follows that the closure of
$\epstl_h(\HH^n)$ is also analytic -- in fact, it is a connected component of the above
set. 

\begin{proposition} \label{prop:closure}
Let $h \in \HZ$ be an element with $N_j h \in W_{-2} \cap e^{\sum w_j N_j} F^{-1}$
for all $j$. Then the image of the holomorphic mapping
\[
	\HH^n \to \Delta^n \times \Hom(F^1, \CC), \quad
		z \mapsto \left( s, v \mapsto 
			Q \bigl( h, e^{\sum z_j N_j} e^{\Gamma(s)} v \bigr) \right)
\]
has an analytic closure (where $s_j = e^{2 \pi i z_j}$ as usual).
\end{proposition}

\begin{proof}
Let $\HC = \bigoplus_{p,q} I^{p,q}$ be Deligne's decomposition of the mixed Hodge
structure $\bigl( W, e^{\sum w_j N_j} F \bigr)$. Since
\[
	e^{\sum z_j N_j} e^{\Gamma(s)} = e^{\sum (z_j - w_j) N_j} 
		\left( e^{\sum w_j N_j} e^{\Gamma(s)} e^{-\sum w_j N_j} \right)
			e^{\sum w_j N_j},
\]
we may replace $F$ by $e^{\sum w_j N_j} F$ and $e^{\Gamma(s)}$ by the expression in
parentheses, and assume without essential loss of generality that $w = 0$. We then
have $N_j h \in I^{-1,-1}$ for every $j = 1, \dotsc, n$. Under the isomorphism
\[
	Q \colon \bigoplus_{p \leq -1} I^{p,q} \to \Hom(F^1, \CC)
\]
induced by $Q$, the linear functional $v \mapsto Q \bigl( h, e^{\sum z_j N_j}
e^{\Gamma(s)} v \bigr)$ corresponds to 
\[
	\sum_{p \leq -1} \Bigl( e^{-\Gamma(s)} e^{-\sum z_j N_j} h \Bigr)^{p,q}
	= \sum_{p \leq -1} \Bigl( e^{-\Gamma(s)} h \Bigr)^{p,q} 
		+ e^{-\Gamma(s)} \Bigl( e^{-\sum z_j N_j} - \id \Bigr) h.
\]
Here we have used the fact that $N_j h \in I^{-1,-1}$ and $\Gamma(s) \in \bigoplus_{p
\leq -1} \glie^{p,q}$. Since $\Gamma(s)$ is holomorphic on $\Delta^n$, it is
therefore enough to prove that the image of
\begin{equation} \label{eq:closure-map}
	\CC^n \to \CC^n \times \bigoplus_{p \leq -1} I^{p,p}, \quad
		z \mapsto \Bigl( e^{2 \pi i z_1}, \dotsc, e^{2 \pi i z_n}, 
			\bigl( e^{-\sum z_j N_j} - \id \bigr) h \Bigr)
\end{equation}
has an analytic closure. This is what we are going to do next. We denote by $V$ the direct
sum of the $I^{p,p}$ with $p \leq -1$; given a vector $v \in V$, we write $v^{p,p}$
for its component in the summand $I^{p,p}$.

Let $S(h) = \menge{a \in \ZZ^n}{\sum a_j N_j h = 0}$; note that $\ZZ^n / S(h)$ embeds
into $\HQ$, and is therefore a free $\ZZ$-module, say of rank $r$. We can thus find a
matrix $A \in \SL_n(\ZZ)$ whose last $n-r$ columns give a basis for the submodule
$S(h) \subseteq \ZZ^n$. If we now introduce new coordinates $(z_1', \dotsc,
z_n') \in \CC^n$ by defining
\[
	z_j = \sum_{k=1}^n a_{j,k} z_k' \quad \text{and} \quad
		N_k' = \sum_{j=1}^n a_{j,k} N_j,
\]
we have $z_1 N_1 + \dotsb + z_n N_n = z_1' N_1' + \dotsb + z_n' N_n'$. The vectors
$N_1' h, \dotsc, N_r' h \in I^{-1,-1}$ are linearly independent, while $N_{r+1}' h =
\dotsb = N_n' h = 0$. The mapping in \eqref{eq:closure-map} therefore has the
same image as
\[
	\CC^n \to \CC^n \times V, \quad 
		z' \mapsto \left( \prod_{k=1}^n e^{2 \pi i a_{1,k} z_k'}, \dotsc, 
		\prod_{k=1}^n e^{2 \pi i a_{n,k} z_k'}, 
		\bigl( e^{\sum z_k' N_k'} - \id \bigr) h \right).
\]
Because $N_1' h, \dotsc, N_r' h$ are linearly independent, we
can find linear functionals $\varphi_1, \dotsc, \varphi_r \colon I^{-1,-1} \to \CC$
with the property that
\[
	z_j' = \varphi_j(z_1' N_1' h + \dotsb + z_r' N_r' h).
\]	
Every point $(s, v) \in \CC^n \times V$ in the image therefore satisfies the
polynomial equations
\begin{equation} \label{eq:closure-1}
	v^{-p,-p} = \frac{(-1)^p}{p!} \Bigl( \varphi_1(v^{-1,-1}) N_1' + \dotsb +
		\varphi_r(v^{-1,-1}) N_r' \Bigr)^p h
\end{equation}
for every $p \geq 1$. Moreover, one has
\[
	s_j = \prod_{k=1}^n e^{2 \pi i a_{j,k} z_k'} 
		= \prod_{k=1}^r e^{2 \pi i a_{j,k} \varphi_k(v^{-1,-1})} \cdot
		\prod_{k=r+1}^n \bigl( e^{2 \pi i z_k'} \bigr)^{a_{j,k}}.
\]
Now observe that the closure of the image of the monomial mapping
\[
	(\CCst)^{n-r} \to \CC^n, \quad
		(t_{r+1}, \dotsc, t_n) \mapsto \left( \prod_{k=r+1}^n t_k^{a_{1,k}}, 
			\dotsc, \prod_{k=r+1}^n t_k^{a_{n,k}} \right)
\]
is a closed algebraic subvariety $Z \subseteq \CC^n$, defined by
finitely many polynomial equations. In fact, each of these polynomials can be taken
as a difference of two monomials, and $Z$ is a (possibly not normal) toric variety.
For every polynomial $f(s)$ in the ideal of $Z$, we therefore obtain a holomorphic
equation
\begin{equation} \label{eq:closure-2}
	f \left( s_1 \cdot \prod_{k=1}^r e^{-2 \pi i a_{1,k} \varphi_k(v^{-1,-1})}, \dotsc,
			s_n \cdot \prod_{k=1}^r e^{-2 \pi i a_{n,k} \varphi_k(v^{-1,-1})} \right) = 0
\end{equation}
that is satisfied by every point $(s,v) \in \CC^n \times V$ in the image.
Together, \eqref{eq:closure-1} and \eqref{eq:closure-2} give a system of holomorphic
equations for the closure of the image of the mapping in \eqref{eq:closure-map},
proving that it is indeed an analytic subvariety.
\end{proof}

\subsection{Another reformulation of the problem}

The purpose of this section is to reduce the proof of \theoremref{thm:local1}
to a statement that only involves the variation of Hodge structure $\shH$.
Suppose then that we are given a sequence of points 
\[
	\bigl( z(m), h(m) \bigr) \in \HH^n \times \HZ(K)
\]
with the properties listed in \theoremref{thm:local1}. Observe that $h \in \HZ$ is a
Hodge class in the Hodge structure defined by $\Phi(z)$ if and only if $h \in
\Phi^0(z)$ if and only if $\epstl(z, h) = 0$.
This suggests that the boundedness of $\epstl \bigl( z(m), h(m) \bigr)$ should allow us
to control the distance from $h(m)$ to the subspace $\Phi^0 \bigl( z(m) \bigr)$. To
quantify this idea, we fix an inner product on $\HC$, and let
$\norm{\argbl}$ denote the corresponding norm. 

\begin{definition}
A sequence of vectors $b(m) \in \HC$ is called \define{harmless} with respect
to $y(m)$ if there is a positive real number $\alpha > 0$ such that the quantity
\[
	\norm{b(m)} + \sum_{k=1}^n e^{\alpha y_k(m)} \norm{N_k b(m)}
\]
remains bounded as $m \to \infty$.
\end{definition}

In other words, a sequence $b(m) \in \HC$ is harmless if and only if it is bounded and
$\norm{N_k b(m)}$ is in $O(e^{-\alpha y_k(m)})$ for every
$k = 1, \dotsc, n$. We can now turn the idea from above into a precise statement.

\begin{proposition} \label{prop:bounded-harmless}
If the sequence $\epstl \bigl( z(m), h(m) \bigr) \in T(F_{-1} \Mmod)$ is bounded, then
\begin{equation} \label{eq:h-and-b}
	h(m) \equiv b(m) \mod \Phi^0 \bigl( z(m) \bigr),
\end{equation}
for a sequence of vectors $b(m) \in \HC$ that is harmless with respect to $\Im z(m)$.
\end{proposition}

\begin{proof}
We are going to use the collection of holomorphic sections
\begin{align*}
	\sigma_v(z) = e^{\sum z_j N_j} e^{\Gamma(s)} v
		\quad \text{(for $v \in F^1$)}  \\
	\sigma_{v,k}(z) = e^{\sum z_j N_j} e^{\Gamma(s)} \frac{N_k v}{s_k}
		\quad \text{(for $v \in F^2$)}
\end{align*}
of the coherent sheaf $F_{-1} \Mmod$. Define the auxiliary sequence of vectors
\[
	h'(m) = e^{-\Gamma(s(m))} e^{-\sum z_j(m) N_j} h(m) \in \HC.
\]
Using Deligne's decomposition $\HC = \bigoplus_{p,q} I^{p,q}$ of the mixed Hodge
structure $(W, F)$, we also define
\[
	h'(m)_{-1} = \sum_{p \leq -1} h'(m)^{p,q} \in \bigoplus_{p \leq -1} I^{p,q}.
\]
We have $h'(m) \equiv h'(m)_{-1}$ modulo $F^0 = \bigoplus_{p \geq 0} I^{p,q}$, 
and therefore \eqref{eq:h-and-b} holds with
\[
	b(m) = e^{\sum z_j(m) N_j} e^{\Gamma(s(m))} h'(m)_{-1} \in \HC.
\]
It remains to show that $b(m)$ is harmless with respect to the sequence
of imaginary parts $y(m) = \Im z(m)$. By assumption, the sequence of complex numbers
\[
	Q \Bigl( h(m), \sigma_v \bigl( z(m) \bigr) \Bigr)
	= Q \left( h(m), e^{\sum z_j(m) N_j} e^{\Gamma(s(m))} v \right)
	= Q \bigl( h'(m), v \bigr)
\]
is bounded for every $v \in F^1$. Since the pairing $Q$ is nondegenerate and
compatible with Deligne's decomposition, we conclude that $\norm{h'(m)_{-1}}$ is
bounded. Likewise, the boundedness of the sequence
\[
	Q \Bigl( h(m), \sigma_{v,k} \bigl( z(m) \bigr) \Bigr)
	= - Q \left( \frac{N_k h'(m)}{s_k(m)}, v \right)
\]
for every $v \in F^2$ implies that $\norm{N_k h'(m)_{-1}}$ is in $O(e^{-2 \pi
y_k(m)})$. Combining both observations, we find that the sequence $h'(m)_{-1}$ is
harmless (with $\alpha = 2 \pi$). But then $b(m)$ is also harmless (with $\alpha < 2
\pi$) by \lemmaref{lem:harmless-operators} below.
\end{proof}

To summarize, we have reduced \theoremref{thm:local1} to the following slightly
stronger statement. It has the advantage of being expressed purely in terms of $\shH$.

\begin{theorem} \label{thm:local2}
Suppose we are given a sequence of points 
\[
	\bigl( z(m), h(m) \bigr) \in \HH^n \times \HZ(K)
\]
with $x_j(m) = \Re z_j(m)$ bounded and $y_j(m) = \Im z_j(m)$ going to infinity for 
every $j = 1, \dotsc, n$. Also suppose that
\[
	h(m) \equiv b(m) \mod \Phi^0 \bigl( z(m) \bigr)
\]
for a sequence of vectors $b(m) \in \HC$ that is harmless with respect to $y(m)$.
Then there exists a subsequence with the following properties: 
\begin{aenumerate}
\item The sequence $h(m)$ is constant and equal to $h \in \HZ(K)$.
\item There are positive integers $a_1, a_2, \dotsc, a_n$ such that 
\[
	(a_1 N_1 + \dotsb + a_n N_n) h = 0.
\]
\item There is a vector $w \in \CC^n$ such that
\[
	\lim_{m \to \infty} e^{-\sum z_j(m) N_j} h(m) = e^{-\sum w_j N_j} h.
\]
\end{aenumerate}
\end{theorem}

Here is why \theoremref{thm:local2} implies \theoremref{thm:local1}.
Let $\bigl( z(m), h(m) \bigr)$ be a sequence with the
properties described in \theoremref{thm:local1}. According to
\propositionref{prop:bounded-harmless}, we have $h(m) \equiv b(m) \mod \Phi^0 \bigl( z(m)
\bigr)$ for a harmless sequence $b(m) \in \HC$; the first three assertions in
\theoremref{thm:local1} therefore follow immediately from \theoremref{thm:local2}. To
prove the fourth one, note that $h \in W_0$. Since we also know that
\[
	Q \left( h(m), e^{\sum z_j(m) N_j} e^{\Gamma(s(m))} \frac{N_k v}{s_k(m)} \right)
\]
is bounded for every choice of $v \in F^2$, we can pass to the limit and obtain
\[
	0 = \lim_{m \to \infty} Q \left( e^{-\sum z_j(m) N_j} h(m), 
		e^{\Gamma(s(m))} N_k v \right) 
	= Q \left( e^{-\sum w_j N_j} h, N_k v \right).
\]
Now the properties of $Q$ imply that $N_k h \in e^{\sum w_j N_j} F^{-1}$.

\subsection{Properties of harmless sequences}

This section contains a few elementary results about harmless sequences that will be
useful later.  First, we prove the following structure theorem; to simplify the
notation, we define $y_{n+1}(m) = 0$.

\begin{proposition} \label{prop:harmless}
A harmless sequence can always be written in the form
\[
	b(m) = b_0(m) + b_1(m) + \dotsb + b_n(m),
\]
where $b_k(m) \in \ker N_1 \cap \dotsb \cap \ker N_k$ and $\norm{b_k(m)}$ is in
$O(e^{-\alpha y_{k+1}(m)})$.
\end{proposition}

In other words, $b_0(m)$ is of size $e^{-\alpha y_1(m)}$; $b_1(m)$ is in the kernel
of $N_1$ and of size $e^{-\alpha y_2(m)}$; and so on, down to $b_n(m)$, which is in the
kernel of all the $N_j$ and bounded. The proof is based on the following simple
result from linear algebra.

\begin{lemma}
Let $T \colon V \to V$ be a linear operator on a finite-dimensional vector space.
Then every $v \in V$ can be written in the form $v = v_0 + v_1$, where $T v_1 = 0$
and $\norm{v_0} \leq C \norm{Tv}$, for a constant $C$ that depends only on $V$, $T$,
and $\norm{\argbl}$.
\end{lemma}

\begin{proof}
Recall that $\norm{\argbl}$ comes from an inner product on $V$. By projecting to
$\ker T$, we get $v = v_0 + v_1$, with $T v_1 = 0$ and $v_0 \perp \ker T$. In
particular, $Tv = T v_0$. Now
\[
	T \colon (\ker T)^{\perp} \to \im T
\]
is an isomorphism, and therefore has an inverse $S$. We then get
\[
	\norm{v_0} = \norm{S(Tv)} \leq C \norm{Tv},
\]
for a constant $C$ that depends only on $V$, $T$, and the choice of norm.
\end{proof}

\begin{proof}[Proof of \propositionref{prop:harmless}]
Since $N_1, \dotsc, N_n$ commute with each other, we can use the lemma and induction.
First, we apply the lemma for $V = \HC$ and $T = N_1$; this gives $b(m) = b_0(m) +
b'(m)$, with $N_1 b'(m) = 0$ and $\norm{b_0(m)}$ in $O(e^{-\alpha y_1(m)})$. In the
next step, we apply the lemma for $V = \ker N_1$ and $T = N_2$ to decompose $b'(m)$,
etc.
\end{proof}

We also need to know that harmless sequences are preserved when we apply certain
operators; this fact has already been used during the proof of
\propositionref{prop:bounded-harmless}.

\begin{lemma} \label{lem:harmless-operators}
Let $\Phi(z) = e^{\sum z_j N_j} e^{\Gamma(s)} F$ be the normal form of a period
mapping on $\HH^n$. If $b(m) \in \HC$ is harmless with respect to $\Im z(m)$, then so
are
\[
	e^{\sum z_j(m) N_j} b(m) \quad \text{and} \quad e^{\Gamma(s(m))} b(m),
\]
provided that $\Im z_1(m), \dotsc, \Im z_n(m)$ are going to infinity.
\end{lemma}

\begin{proof}
Since the operator $e^{\sum z_j N_j}$ is polynomial in $z_1, \dotsc, z_n$, whereas
$\norm{N_j b(m)}$ is in $O(e^{-\alpha \Im z_j(m)})$, it is clear that $e^{\sum z_j(m)
N_j} b(m)$ is again harmless (for a slightly smaller value of $\alpha$). On the other
hand, the operator $e^{\Gamma(s)}$ is holomorphic on $\Delta^n$, and therefore
bounded; moreover, \propositionref{prop:Gamma} shows that the norm of
\[
	N_j e^{\Gamma(s(m))} b(m) - e^{\Gamma(s(m))} N_j b(m)
\]
is bounded by a constant multiple of $\abs{s_j(m)} = e^{-2 \pi \Im z_j(m)}$.
This is clearly enough to conclude that $e^{\Gamma(s(m))} b(m)$ is a harmless
sequence, too.
\end{proof}

Next, we consider the case when the sequence $h(m)$ belongs to certain subspaces. For
any subset $J \subseteq \{1, \dotsc, n\}$, we let $W(J)$ denote the weight filtration
of the cone
\[
	C(J) = \mengge{\sum_{j \in J} a_j N_j}{\text{$a_j > 0$ for every $j \in J$}}.
\]
We would like to know that when $h(m) \in W_w(J)$, we can also take $b(m) \in
W_w(J)$. This requires the following assumption on the period mapping.

\begin{definition} \label{def:nilpotent-J}
Let $J \subseteq \{1, \dotsc, n\}$ be a subset of the index set. We say that
$\Phi(z) = e^{\sum z_j N_j} e^{\Gamma(s)} F$ is a \define{nilpotent orbit in the
variables $\{s_j\}_{j \in J}$} if 
\[
	\frac{\partial \Gamma(s)}{\partial s_j} = 0 \quad \text{for every $j \in J$;}
\]
in other words, if $\Gamma(s)$ does not depend on the variables $\{s_j\}_{j \in J}$.
\end{definition}

The point is that $\Gamma(s)$ then commutes with $N_j$ for $j \in J$ (by
\propositionref{prop:Gamma}), and therefore preserves the weight filtration $W(J)$.
Note that nilpotent orbits in the usual sense are the special case when $J = \{1,
\dotsc, n\}$.

\begin{lemma} \label{lem:harmless-W}
Suppose that $h(m) \equiv b(m) \mod \Phi^0 \bigl( z(m) \bigr)$, with $b(m) \in \HC$
harmless. If $h(m) \in W_w(J)$, and if $\Phi(z)$ is a nilpotent
orbit in the variables $\{s_j\}_{j \in J}$, then one can arrange that $b(m) \in
W_w(J)$ as well.
\end{lemma}

\begin{proof}
This follows from the proof of \propositionref{prop:bounded-harmless}. Because of the
assumption on $\Phi(z)$, the function $\Gamma(s)$ commutes with $N_j$ for $j \in J$
(by \propositionref{prop:Gamma}), and therefore preserves the weight filtration
$W(J)$. Consequently, 
\[
	h'(m) = e^{-\Gamma(s(m))} e^{-\sum z_j(m) N_j} h(m) \in W_w(J).
\]
Because each $N_j$ is a $(-1,-1)$-morphism, $W_w(J)$ is a sub-mixed Hodge
structure of $(W,F)$; it is therefore compatible with Deligne's decomposition, and so
\[
	h'(m)_{-1} = \sum_{p \leq -1} h'(m)^{p,q} \in W_w(J).
\]
At the same time, we clearly have $h'(m) \equiv b'(m) \mod F^0$, which means that
$h'(m)_{-1} = b'(m)_{-1}$ is a harmless sequence. We conclude as before from
\propositionref{prop:bounded-harmless} that $e^{\sum z_j N_j} e^{\Gamma(s)}
h'(m)_{-1} \in W_w(J)$ is harmless.
\end{proof}

We end this section with a word of caution. During the proof of \theoremref{thm:local2}, the
fact that $b(m) \not\in \HR$ causes some trouble. If being harmless was preserved by
the Hodge decomposition for the Hodge structure $\Phi \bigl( z(m) \bigr)$, we could
easily arrange that $b(m) \in \HR$. Unfortunately, this is not the case.

\begin{example}
Let $n = 1$, and consider the special case $\HC = I^{1,1} \oplus
I^{-1,-1}$.  If $b \in I^{-1,-1}$, then $N b = 0$, and so $b$ is harmless. Now we
decompose with respect to 
\[
	e^{i y N} F^0 \oplus e^{-i y N} \overline{F^1}.
\]
A short calculation gives
\[
	b = e^{i y N} \frac{\Np b}{2 i y} - e^{-i y N} \frac{\Np b}{2 i y}
		= \left( \frac{b}{2} + \frac{\Np b}{2 i y} \right) 
			+ \left( \frac{b}{2} - \frac{\Np b}{2 i y} \right) 
\]
and neither of the two components is harmless. The best one can say is that, 
after applying $yN$, they are bounded; this is consistent with
\cite[Proposition~24.3]{Schnell-N}.
\end{example}

\subsection{Proof in the one-dimensional case}
\label{par:one-dim}

In this section, we prove \theoremref{thm:local2} in the special case $n = 1$. This
case is technically easier, because it avoids the complications coming from the
presence of several variables and several nilpotent operators. Because many key
features of the proof are the same as in the general case, it may be useful to
understand them first in this special case.

Suppose then that $\bigl( z(m), h(m) \bigr) \in \HH \times \HZ(K)$ is a sequence of
the type considered in \theoremref{thm:local2}. Fix an inner product on the
space $\HC$, and denote by $\norm{\argbl}$ the corresponding norm. By
\propositionref{prop:harmless}, we can arrange that
\[
	h(m) \equiv b(m) = b_0(m) + b_1(m) \mod \Phi^0 \bigl( z(m) \bigr),
\]
with $\norm{b_0(m)}$ in $O(e^{-\alpha y(m)})$, and $b_1(m) \in \ker N$ bounded. Our
goal is to prove that, after taking a subsequence, $h(m)$ is constant and $N h(m) =
0$.

We first introduce some notation. Let $(W, \Fh)$ denote the $\RR$-split mixed Hodge
structure canonically associated with $(W, F)$ by the $\SL(2)$-orbit theorem
\cite[Corollary~3.15]{CK}. If $Y \in \glieR$ denotes the corresponding splitting, the
eigenspaces $E_{\ell}(Y)$ define a real grading of the weight filtration $W$, meaning
that
\[
	W_k = \bigoplus_{\ell \leq k} E_{\ell}(Y). 
\]
To simplify some of the arguments below, we shall choose the inner product on 
$\HC$ in such a way that this decomposition is orthogonal. The most important tool in
the proof will be the following sequence of real operators:
\[
	e(m) = \exp \left( \frac{1}{2} \log y(m) \cdot Y \right) \in \End(\HR)
\]
Note that $e(m)$ acts as multiplication by $y(m)^{\ell/2}$ on the subspace
$E_{\ell}(Y)$, and preserves the filtration $\Fh$. Because $\lie{Y}{N} = - 2 N$, we
have
\begin{equation} \label{eq:e-N}
	e(m) N = \frac{1}{y(m)} N e(m).
\end{equation}
Since the sequence of real parts $x(m)$ is bounded, \cite[Theorem~4.8]{CK} shows that
\begin{equation} \label{eq:Fsh}
	\Fsh \defeq e^{iN} \Fh = \lim_{m \to \infty} e(m) \Phi \bigl( z(m) \bigr) \in D.
\end{equation}
The filtration $\Fsh$ has two important properties: on the one hand, it belongs to
$D$, and therefore defines a polarized Hodge structure of weight zero on $\HC$; on
the other hand, the pair $(W, \Fsh)$ is a mixed Hodge structure.

We divide the proof of the theorem (in the case $n = 1$) into six steps; each of the six
steps will appear again in a similar form during the proof of the general case.

\paragraph{Step 1}
We prove that $b(m)$ and $h(m)$ are bounded with respect to the Hodge norm. This will
also show that $\norm{h(m)}$ grows at most polynomially in $y(m)$. 

\begin{lemma} \label{lem:HodgeNorm-1}
The two sequences $\norm{b(m)}_{\Phi(z(m))}$ and $\norm{h(m)}_{\Phi(z(m))}$ are bounded. 
\end{lemma}

\begin{proof}
Recall that the Hodge norm of a vector $h \in \HC$ with respect to the polarized
Hodge structure $\Phi(z) \in D$ is defined as
\[
	\norm{h}_{\Phi(z)}^2 = \sum_{p \in \ZZ} \norm{h^{p,-p}}_{\Phi(z)}^2 
		= \sum_{p \in \ZZ} (-1)^p Q \bigl( h^{p,-p}, \widebar{h^{p,-p}} \bigr),
\]
where $h = \sum h^{p,-p}$ is the Hodge decomposition of $h$ in $\Phi(z)$. We begin the
proof by observing that the sequence $e(m) b(m) = e(m) b_0(m) + e(m) b_1(m)$
is bounded, for the following reason. On the one hand, $b_1(m) \in \ker N \subseteq
W_0$ implies that
\[
	b_1(m) = \sum_{\ell \leq 0} b_1(m)_{\ell} 
		\in \bigoplus_{\ell \leq 0} E_{\ell}(Y);
\]
consequently, the boundedness of $b_1(m)$ implies the boundedness of
\[
	e(m) b_1(m) = \sum_{\ell \leq 0} y(m)^{\ell/2} b_1(m)_{\ell}.
\]
On the other hand, the term $e(m) b_0(m)$ is going to zero, because
$\norm{b_0(m)}$ is in $O(e^{-\alpha y(m)})$, whereas $e(m)$ grows at most
polynomially in $y(m)$. Because $e(m)$ is a real operator, we have
\[
	\norm{b(m)}_{\Phi(z(m))} = \norm{e(m) b(m)}_{e(m) \Phi(z(m))},
\]
which is bounded by virtue of \eqref{eq:Fsh}. Now let 
\[
	h(m) = \sum_{p \in \ZZ} h(m)^{p,-p}
\]
denote the Hodge decomposition of $h(m)$ in the Hodge structure $\Phi(z(m))$. The
difference $h(m) - b(m)$ is an element of $\Phi^0(z(m))$, and for $p \leq -1$, 
\[
	\norm{h(m)^{p,-p}}_{\Phi(z(m))} = \norm{b(m)^{p,-p}}_{\Phi(z(m))}
\]
is bounded. Recalling that $h(m) \in \HZ(K)$, we now have
\begin{align*}
	\norm{h(m)}_{\Phi(z(m))}^2 &= Q \bigl( h(m), h(m) \bigr) + 
		\sum_{p \neq 0} \bigl( 1 - (-1)^p \bigr) \norm{h(m)^{p,-p}}_{\Phi(z(m))}^2 \\
		&\leq K + 4 \norm{b(m)}_{\Phi(z(m))}^2,
\end{align*}
and so the Hodge norm of $h(m)$ is bounded, too.
\end{proof}

\paragraph{Step 2}
Next, we reduce the problem to the case of a nilpotent orbit.
\lemmaref{lem:HodgeNorm-1} gives the boundedness of the sequence $e(m) h(m)$. Since
$e(m)^{-1}$ is polynomial in $y(m)$, it follows that $\norm{h(m)}$ grows at most like a
fixed power of $y(m)$. We have 
\[
	e^{z(m) N} e^{-\Gamma(s(m))} e^{-z(m) N} \bigl( h(m) - b(m) \bigr) 
		\in e^{z(m) N} F^0,
\]
and by using \lemmaref{lem:harmless-operators} and the bounds on $\norm{h(m)}$ and
$\norm{b(m)}$, we see that $h(m)$ is congruent to a harmless sequence modulo
$e^{z(m) N} F^0$. We may therefore assume without loss of generality that $\Phi(z) =
e^{z N} F$ is a nilpotent orbit. Note that we only have $\Phi(z) \in D$ when the
imaginary part of $z \in \HH$ is sufficiently large; this does not cause any problems
because $\Im z(m)$ is going to infinity anyway.

\paragraph{Step 3}
We exploit the boundedness of $e(m) h(m)$ to prove that $h(m) \in W_0$. By passing
to a subsequence, we can arrange that there is a limit
\[
	v = \lim_{m \to \infty} e(m) h(m) \in \HR.
\]
With respect to the eigenspace decomposition of $Y$, we have
\[
	e(m) h(m) = \sum_{\ell \in \ZZ} y(m)^{\ell/2} h(m)_{\ell},
\]
and so $h(m)_{\ell}$ is going to zero when $\ell \geq 1$, and is bounded when $\ell =
0$. Let $\ell \in \ZZ$ be the largest index such that $h(m)_{\ell} \neq 0$ along a
subsequence. The projection from $E_{\ell}(Y)$ to $\gr_{\ell}^W$ is an isomorphism,
and because $h(m) \in \HZ$, it follows that $h(m)_{\ell}$ lies in a discrete
subset of $E_{\ell}(Y)$. This is only possible if $\ell \leq 0$, and hence $h(m) \in
W_0$; moreover, the component $h(m)_0$ takes values in a finite set. After passing to
a subsequence, we may therefore assume that
\[
	h(m) \equiv h_0 \mod W_{-1},
\]
where $h_0 \in E_0(Y)$ is constant.

\paragraph{Step 4}
We prove that $N h_0 = 0$. Consider again the decomposition
\[
	b_1(m) = \sum_{\ell \leq 0} b_1(m)_{\ell},
\]
where $b_1(m)_{\ell} \in E_{\ell}(Y) \cap \ker N$; note that all summands are
bounded, and that only terms with $\ell \leq 0$ appear because $b_1(m) \in \ker N
\subseteq W_0$. Consequently,
\[
	e(m) b_1(m) = b_1(m)_0 + \sum_{\ell \leq -1} y(m)^{\ell/2} b_1(m)_{\ell}
\]
has the same limit as $b_1(m)_0 \in E_0(Y) \cap \ker N$. If we now look back at
\[
	e(m) h(m) \equiv e(m) b_0(m) + e(m) b_1(m) \mod e(m) \Phi^0 \bigl( z(m) \bigr),
\]
we find that all terms converge individually, and hence that
\[
	v = \lim_{m \to \infty} e(m) h(m) \equiv \lim_{m \to \infty} b_1(m)_0  \mod \Fsh^0.
\]
To show that $N v = 0$, we apply the following version of
\cite[Proposition~3.10]{CDK} to the $\RR$-split mixed Hodge structure $(W, \Fh)$,
recalling that $\Fsh = e^{iN} \Fh$. This result is, in a sense, an asymptotic form
of \theoremref{thm:local2}.

\begin{lemma} \label{lem:MHS}
Let $(W, F)$ be an $\RR$-split mixed Hodge structure on $\HC$, and let $N$ be a real
$(-1,-1)$-morphism of $(W, F)$. If a vector $h \in W_{2\ell} \cap \HR$ satisfies 
\[
	h \equiv b \mod e^{iN} F^{\ell}
\]
for some $b \in E_{2\ell}(Y) \cap \ker N$, then $h \in E_{2\ell}(Y) \cap \ker N$.
\end{lemma}

\begin{proof}
If we apply $e^{-iN}$ to both sides and use the fact that $N b = 0$, we get
\[
	e^{-iN} h \in E_{2\ell}(Y) + W_{2\ell} \cap F^{\ell}.
\]
Thus $Nh \in W_{2\ell-2} \cap F^{\ell} \cap \HR$, which can only happen if $Nh = 0$,
because $(W, F)$ is a mixed Hodge structure. But then $Yh - 2\ell h \in W_{2\ell-1} \cap
F^{\ell} \cap \HR$, and so $Yh = 2\ell h$ for the same reason. Alternatively, one
can use the decomposition into the subspaces $I^{p,q} = W_{p+q} \cap F^p \cap
\widebar{F^q}$, which is preserved by $Y$.
\end{proof}

Consequently, $v \in E_0(Y) \cap \ker N$. We can now project the congruence
\[
	e(m) h(m) \equiv h_0 \mod W_{-1}
\]
to the subspace $E_0(Y)$ to conclude that $v = h_0$, and hence that $N h_0 = 0$.

\paragraph{Step 5}
Next, we shall argue that $N h(m) = 0$. Since $N h_0 = 0$, we already know that
$N h(m) \in W_{-3}$. In addition, we have $N b_1(m) = 0$, and so
\[
	e(m) N h(m) \equiv e(m) N b_0(m) \mod N e(m) \Phi^0 \bigl( z(m) \bigr);
\]
here we have used \eqref{eq:e-N} to interchange $N$ and $e(m)$. 

We claim that $\norm{e(m) N h(m)}$ is bounded by a constant multiple of $\norm{e(m) N
b_0(m)}$, and therefore in $O(e^{-\alpha y(m)})$. If not, then the ratios
\[
	\frac{\norm{e(m) N b_0(m)}}{\norm{e(m) N h(m)}}
\]
would be going to zero; after passing to a subsequence, the unit vectors
\[
	\frac{e(m) N h(m)}{\norm{e(m) N h(m)}} \in W_{-3} \cap \HR
\]
would then converge to a unit vector in $W_{-3} \cap N \Fsh^0 \cap \HR \subseteq
W_{-3} \cap \Fsh^{-1} \cap \HR$; but this is not possible because $(W, \Fsh)$ is a
mixed Hodge structure. Consequently, $\norm{e(m) N h(m)}$ is in $O(e^{-\alpha y(m)})$; 
because $e(m)^{-1}$ only grows like a power of $y(m)$, the same is true for the norm
of $N h(m)$. But these vectors lie in a discrete set, and so $N h(m) = 0$ after
passing to a subsequence.

\paragraph{Step 6}
To finish the proof, we have to show that $h(m)$ is bounded. Choose a point $w \in
\HH$ with sufficiently large imaginary part to ensure that $\Phi(w) = e^{w N} F \in
D$; this filtration defines a polarized Hodge structure of weight zero on $\HC$.
Since $\Phi(z)$ is a nilpotent orbit, we then have
\[
	h(m) = e^{w N - z(m) N} h(m) \equiv 
		e^{w N - z(m) N} b_0(m) + b_1(m) \mod \Phi^0(w);
\]
note that both terms on the right-hand side are bounded. This relation shows that,
with respect to the Hodge structure $\Phi(w)$, all the Hodge components
$h(m)^{p,-p}$ with $p \leq -1$ are bounded. Because $h(m) \in \HZ(K)$, we conclude as
in \lemmaref{lem:HodgeNorm-1} that $h(m)$ is bounded in the Hodge norm for $\Phi(w)$,
and therefore bounded. After passing to a subsequence, the sequence $h(m)$ becomes
constant. This completes the proof of \theoremref{thm:local2} in the case $n = 1$.

\subsection{Setup in the general case}

In the remaining sections, we shall prove \theoremref{thm:local2} in general, by
adapting the method in \cite[Section~4]{CDK} to our setting. The proof uses many
results from the theory of degenerating variations of Hodge structures; these will be
introduced in the appropriate places. Thoughout the discussion, we fix an inner
product on $\HC$, and denote by $\norm{\argbl}$ the corresponding norm; we shall make
a more specific choice later on. 

To explain the idea of the proof, let us first consider a sequence $z(m) \in
\HH^n$, with the property that the real parts $x_j(m) = \Re z_j(m)$ are bounded, and
the imaginary parts $y_j(m) = \Im z_j(m)$ are going to infinity.  In the course of
the argument, it will often be necessary to pass to a subsequence; to reduce clutter,
we shall use the same notation for the subsequence. A new feature of the general
case is that we no longer have a unique scale on which we can measure the rate of
growth of a sequence; the reason is that $y_1(m), \dotsc, y_n(m)$ may be going to
infinity at different rates. The most efficient way to deal with this problem is as
follows.

Following \cite[(4.1.3)]{CDK}, we expand the sequence $z(m)$ according to the
rate of growth of its imaginary parts. After passing to a subsequence, we can find an
integer $1 \leq d \leq n$ and an $n \times d$-matrix $A$ with nonnegative real
entries, such that 
\begin{equation} \label{eq:z-t}
	z(m) = i A t(m) + w(m).
\end{equation}
Here $w(m) \in \CC^n$ is a convergent sequence, with the property that $\Phi \bigl(
w(m) \bigr) \in D$ for every $m \in \NN$; and $t(m) \in \RR^d$ has the property
that all the ratios
\[
	\frac{t_1(m)}{t_2(m)}, \frac{t_2(m)}{t_3(m)}, \dotsc, \frac{t_d(m)}{t_{d+1}(m)} 
\]
are going to infinity. (To avoid having to deal with special cases, we always define
$t_{d+1}(m) = 1$.) Moreover, we can partition the index set
\[
	\{1, 2, \dotsc, n\} = J_1 \sqcup J_2 \sqcup \dotsb \sqcup J_d
\]
in such a way that $a_{j,k} \neq 0$ if and only if $j \in J_1 \sqcup \dotsb \sqcup
J_k$. By construction, $\abs{s_j(m)}$ is in $O(e^{-\alpha t_k(m)})$ for every $j
\in J_k$. We define new operators
\[
	T_k = \sum_{j=1}^n a_{j,k} N_j \in C(J_1 \sqcup \dotsb \sqcup J_k)
\]
and have the identity
\[
	\sum_{j=1}^n z_j(m) N_j = \sum_{k=1}^d i t_k(m) T_k 
		+ \sum_{j=1}^n w_j(m) N_j,
\]

Now suppose that $b(m) \in \HC$ is a harmless sequence with respect to $\Im z(m)$. 
By definition, there is some $\alpha > 0$ such that
\[
	\norm{b(m)} + \sum_{j=1}^n e^{\alpha y_j(m)} \norm{N_j b(m)}
\]
is bounded; it is easy to see that the same is true (with a different $\alpha > 0$)
for
\[
	\norm{b(m)} + \sum_{k=1}^d e^{\alpha t_j(m)} \norm{T_j b(m)}.
\]
From now on, we shall use the expression \define{harmless} or \define{harmless with
respect to $t(m)$} to refer to sequences with this property. 

We will prove the following version of \theoremref{thm:local2}.

\begin{theorem} \label{thm:local3}
Suppose we are given a sequence of points 
\[
	\bigl( z(m), h(m) \bigr) \in \HH^n \times \HZ(K),
\]
where $z(m) = i A t(m) + w(m)$ is as above, and $h(m) \equiv b(m) \mod \Phi^0 \bigl( z(m)
\bigr)$ for a sequence of vectors $b(m) \in \HC$ that is harmless with respect to
$t(m)$.  Then after passing to a subsequence, $h(m)$ becomes constant, and $T_1 h(m)
= \dotsb = T_d h(m) = 0$.
\end{theorem}

\begin{proof}[Proof that \theoremref{thm:local3} implies \theoremref{thm:local2}]
Suppose we are given a sequence of points $\bigl( z(m), h(m) \bigr)$ as in
\theoremref{thm:local2}. As explained above, we can find a subsequence along which
$z(m) = i A t(m) + w(m)$; after passing to a further subsequence, $h(m)$ is constant
and $T_1 h(m) = \dotsb = T_d h(m) = 0$. Let $h \in \HZ$ denote the constant value. We
have
\[
	T_d h = \sum_{j=1}^n a_{j,d} N_j h = 0
\]
for positive real numbers $a_{1,d}, \dotsc, a_{n,d}$. Since $N_j h \in \HQ$, we can
then obviously find positive integers $a_1, \dotsc, a_n$ with the property that $a_1
N_1 h + \dotsb + a_n N_n h = 0$. At the same time,
\[
	\lim_{m \to \infty} e^{-\sum z_j(m) N_j} h(m) = 
		\lim_{m \to \infty} e^{-\sum w_j(m) N_j} h = e^{-\sum w_j N_j} h,
\]
where $w \in \CC^n$ is the limit of the sequence $w(m)$.
\end{proof}

The proof of \theoremref{thm:local3} is organized as follows. In
\parref{par:bounded}, we introduce a common $\ZZ^d$-grading for the weight
filtrations of $T_1, \dotsc, T_d$, and a corresponding sequence of operators $e(m) \in
\End(\HR)$, and show that the boundedness of $Q \bigl( h(m), h(m)
\bigr)$ is equivalent to the boundedness of the sequence $e(m) h(m) \in \HR$.  In
\parref{par:induction}, we show that the subquotients of the weight filtration
$W(T_1)$ again satisfy the assumptions of the theorem. In \parref{par:position}, we
explain how the boundedness of $e(m) h(m)$ can be used to control the position of the
sequence $h(m)$ with respect to the above $\ZZ^d$-grading. The actual proof of the theorem
will be given in \parref{par:proof-n}.

\subsection{Boundedness results}
\label{par:bounded}

The purpose of this section is to translate the boundedness of $Q \bigl( h(m), h(m)
\bigr)$ into a more manageable condition. Since the existence of an integral
structure is not important here, we consider an arbitrary sequence of real vectors
$h(m) \in \HR$, subject only to the condition that
\begin{equation} \label{eq:assumption}
	h(m) \equiv b(m) \mod \Phi^0 \bigl( z(m) \bigr)
\end{equation}
for a harmless sequence $b(m) \in \HC$. By \propositionref{prop:harmless}, we have 
\[
	b(m) = b_0(m) + b_1(m) + \dotsb + b_d(m),
\]
where $b_k(m) \in \ker T_1 \cap \dotsb \cap \ker T_k$ is in $O(e^{-\alpha
t_{k+1}(m)})$. By passing to a subsequence, we can also arrange that $b_d(m)$ converges
to an element of $\HC$. 

\begin{proposition} \label{prop:bounded}
Given \eqref{eq:assumption}, the following statements are equivalent:
\begin{enumerate}
\item The sequence $Q \bigl( h(m), h(m) \bigr)$ is bounded
\item The sequence of Hodge norms $\norm{h(m)}_{\Phi(z(m))}$ is bounded.
\item The sequence $e(m) h(m) \in \HR$ is bounded.
\end{enumerate}
If any of them is satisfied, $\norm{h(m)}$ is in $O(t_1(m)^N)$ for some $N \in \NN$.
\end{proposition}

The proof is based on the existence of a $\ZZ^d$-grading on $\HC$ with good
properties. Before we can define it, we have to recall a few results from the theory
of degenerating variations of Hodge structure. Let $W^k = W(T_k)$ denote the weight
filtration of the nilpotent operator $T_k$; it agrees with that of the cone $C(J_1
\sqcup \dotsb \sqcup J_k)$, and is therefore defined over $\QQ$, even though $T_k$ is
only defined over $\RR$. The multi-variable $\SL(2)$-orbit theorem
\cite[Theorem~4.3]{CK} associates with 
\[
	\bigl( W, F, T_1, \dotsc, T_d \bigr)
\]
a sequence of mutually commuting splittings $Y_1, \dotsc, Y_d \in \End(\HR)$.
Their common eigenspaces define a real $\ZZ^d$-grading 
\[
	\HC = \bigoplus_{\ell \in \ZZ^d} \HC^{(\ell_1, \dotsc, \ell_d)}
\]
of the vector space $\HC$ (and also of $\HR$), with the property that
\[
	W_w^k 
	= \bigoplus_{\ell_1 + \dotsb + \ell_k \leq w} \HC^{(\ell_1, \dotsc, \ell_d)}.
\]
Given a vector $h \in \HC$, we denote its component in the subspace $\HC^{(\ell_1,
\dotsc, \ell_d)}$ by the symbol $h^{(\ell_1, \dotsc, \ell_d)}$. To simplify some
arguments below, we shall assume that the norm $\norm{\argbl}$ comes from an inner
product for which the decomposition is orthogonal. As in the one-variable case, we
then define a sequence of operators
\begin{equation} \label{eq:em-def}
	e(m) = \exp \left( \frac{1}{2} \sum_{k=1}^d t_k(m) Y_k \right) \in \End(\HR);
\end{equation}
note that $e(m)$ acts on the subspace $\HC^{(\ell_1, \dotsc, \ell_d)}$ as
multiplication by
\begin{equation} \label{eq:action-e}
\begin{split}
	t_1(m)^{\ell_1/2} &t_2(m)^{\ell_2/2} \dotsm t_d(m)^{\ell_d/2} \\
		&= \left( \frac{t_1(m)}{t_2(m)} \right)^{\ell_1/2} 
		\left( \frac{t_2(m)}{t_3(m)} \right)^{(\ell_1 + \ell_2)/2} \dotsm
		\left( \frac{t_d(m)}{t_{d+1}(m)} \right)^{(\ell_1 + \dotsb + \ell_d)/2}
\end{split}
\end{equation}
What makes these operators useful is that the filtrations $e(m) \Phi \bigl( z(m)
\bigr)$ have a well-defined limit, which is again a polarized Hodge structure. In
other words,
\begin{equation} \label{eq:Fsh-n}
	\Fsh = \lim_{m \to \infty} e(m) \Phi \bigl( z(m) \bigr) \in D.
\end{equation}
This is explained in \cite[Theorem~4.8]{CK}, and depends on the fact that $w(m)$ is
bounded and all the ratios $t_k(m) / t_{k+1}(m)$ are going to infinity. This is one
reason for using an expansion of the form $z(m) = i A t(m) + w(m)$.

The multi-variable $\SL(2)$-orbit theorem gives some additional information about the
filtration $\Fsh$. According to \cite[Theorem~4.3]{CK}, there are nilpotent
operators $\Th_k \in C(J_1 \sqcup \dotsb \sqcup J_k)$, with the property that
\[
	\lie{Y_j}{\Th_k} = \begin{cases}
		- 2 \Th_k &\text{if $j = k$,} \\
			0 &\text{otherwise;}
	\end{cases}
\]
note that $\Th_1 = T_1$. In this notation, each of the $d$ pairs
\[
	\bigl( W^k, e^{-i(\Th_1 + \dotsb + \Th_k)} \Fsh \bigr)
\]
defines an $\RR$-split mixed Hodge structure on $\HC$, whose associated grading is
given by $Y_1 + \dotsb + Y_k$ \cite[Theorem~4.3]{CK}. In particular, every $(W^k,
\Fsh)$ is itself a mixed Hodge structure; this fact will be important later.

We now turn to the proof of \propositionref{prop:bounded}. As in the one-variable
case, we first study the effect of the operator $e(m)$ on the harmless sequence $b(m)$.

\begin{lemma} \label{lem:limit-bm}
Suppose that $b(m) \in \HC$ is a harmless sequence. Then
\[
	\lim_{m \to \infty} e(m) b(m) = \lim_{m \to \infty} b_d(m)^{(0, \dotsc, 0)},
\]
and the limit belongs to $\HC^{(0, \dotsc, 0)} \cap \ker T_1 \cap \dotsb \cap \ker
T_d$.
\end{lemma}

\begin{proof}
Since $b(m)$ is harmless with respect to $t(m)$, it is easy to see that $e(m) b(m)$
is bounded. Indeed, we have
\[
	e(m) b(m) = e(m) b_0(m) + e(m) b_1(m) + \dotsb + e(m) b_d(m).
\]
Now $\norm{b_k(m)}$ is in $O(e^{-\alpha t_{k+1}(m)})$, and so the same is true for each of
the components in the decomposition
\[
	b_k(m) = \sum_{\ell \in \ZZ^d} b_k(m)^{(\ell_1, \dotsc, \ell_d)}.
\]
On the other hand, $b_k(m)$ is in $\ker T_1 \cap \dotsb \cap \ker T_k \subseteq
W_0^1 \cap \dotsb \cap W_0^k$; this means that $b_k(m)^{(\ell_1,
\dotsc, \ell_d)} = 0$ unless $\ell_1 \leq 0$, $\ell_1 + \ell_2 \leq 0$, and so on up
to $\ell_1 + \dotsb + \ell_k \leq 0$. It follows from this and \eqref{eq:action-e}
that
\[
	e(m) b_k(m) = \sum_{\ell \in \ZZ^d} 
		t_1(m)^{\ell_1/2} \dotsm t_d(m)^{\ell_d/2} \cdot b_k(m)^{(\ell_1, \dotsc, \ell_d)}
\]
is going to zero for $k = 0, \dotsc, d-1$, and converges for $k = d$. This implies
the asserted formula for the limit. 
\end{proof}

\begin{proof}[Proof of \propositionref{prop:bounded}]
By the previous lemma, the sequence $e(m) b(m)$ converges, and so
\[
	\norm{b(m)}_{\Phi(z(m))} = \norm{e(m) b(m)}_{e(m) \Phi(z(m))}
\]
is bounded by virtue of \eqref{eq:Fsh-n}. With respect to the Hodge structure $\Phi
\bigl( z(m) \bigr)$,
\[
	\widebar{h(m)^{-p,p}} = h(m)^{p,-p} = b(m)^{p,-p}
\]
for every $p \leq -1$. This gives us a bound on the difference
\[
	\norm{h(m)}_{\Phi(z(m))}^2 - Q \bigl( h(m), h(m) \bigr) 
		= \sum_{p \neq 0} \bigl( 1 - (-1)^p \bigr) \norm{h(m)^{p,-p}}_{\Phi(z(m))}^2;
\]
the boundedness of $Q \bigl( h(m), h(m) \bigr)$ is therefore equivalent to
the boundedness of $\norm{h(m)}_{\Phi(z(m))}$. Because we also have
\[
	\norm{e(m) h(m)}_{e(m) \Phi(z(m))} = \norm{h(m)}_{\Phi(z(m))},
\]
both conditions are equivalent to the boundedness of the sequence $e(m) h(m)$.
The last assertion follows from the fact that the operator $e(m)^{-1}$ depends
polynomially on $t_1(m), \dotsc, t_d(m)$. 
\end{proof}

\subsection{Mechanism of the induction}
\label{par:induction}

The proof of \theoremref{thm:local3} is by induction on $d \geq 1$. One situation
where we can potentially apply the inductive hypothesis is for a subquotient of the form
\[
	\HCtl = \gr_{\ell_1}^{W^1} = W_{\ell_1}^1 / W_{\ell_1 - 1}^1;
\]
the point is that $T_1$ acts trivially on the quotient. In this section, we show that
under a certain assumption on the period mapping $\Phi(z)$, the quotient again
supports a polarized variation of $\ZZ$-Hodge structure of weight $\ell_1$. 

Denote by $\Ftl$ the filtration on $\HCtl$ induced by $F$.  We
also write $\Ntl_j$, $\Ttl_k$, and $\Ytl_k$ for the operators induced by $N_j$,
$T_k$, and $Y_k$, respectively, and $\Wtl^k$ for the filtration induced by $W^k$; then
\[
	\Wtl^k = W(\Ttl_k) \decal{-\ell_1}
\]
because $W^k$ is the relative weight filtration of $T_k$ on $W^1$ by
\cite[Theorem~2.9]{CK}. The operators $\Ytl_2, \dotsc, \Ytl_d$ define a
$\ZZ^{d-1}$-grading on $\HCtl$, which is compatible with the $\ZZ^d$-grading on
$\HC$; in fact, the projection
\[
	\HC^{(\ell_1, \ell_2, \dotsc, \ell_d)} \to \HCtl^{(\ell_2, \dotsc, \ell_d)}
\]
is an isomorphism. As in \eqref{eq:em-def}, we define a sequence of operators
\[
	\etl(m) = \exp \left( \frac{1}{2} \sum_{k=2}^d t_k(m) \Ytl_k \right) \in
		\End(\HRtl).
\]
Finally, let $\norm{\argbl}$ denote the norm on $\HCtl$ induced by the isomorphism
$E_{\ell_1}(Y_1) \simeq \HCtl$.

\begin{proposition} \label{prop:induction}
Suppose that $\Phi(z)$ is a nilpotent orbit in the variables $\{s_j\}_{j \in J_1}$,
in the sense of \definitionref{def:nilpotent-J}. Then
\[
	\Phitl(z) = e^{\sum_{j \not\in J_1} z_j \Ntl_j} e^{\Gammatl(s)} \Ftl
\]
defines a polarized variation of $\ZZ$-Hodge structure of weight $\ell_1$ on $\HCtl$.
\end{proposition}

\begin{proof}
We first explain how $\Gammatl(s)$ is defined. For $j \in J_1$, the operator
$\Gamma(s)$ does not depend on $s_j$, and therefore commutes with $N_j$ by
\propositionref{prop:Gamma}. Consequently, $\Gamma(s)$ preserves the weight
filtration $W^1 = W(J_1)$, and therefore induces a similar operator $\Gammatl(s)$ on
$\HCtl$. By \cite[Proposition~2.10]{CK}, the pair
\[
	\bigl( W^1, e^{\sum_{j \not\in J_1} z_j N_j} e^{\Gamma(s)} \bigr)
\]
is a mixed Hodge structure, polarized by the form $Q$ and every element of the cone
$C(J_1)$, in the sense of \cite[Definition~1.16]{CK}. Choose an arbitrary rational
element in $C(J_1)$; together with the Lefschetz decomposition \cite[(1.11)]{CK} for
this element, $Q$ gives rise to a bilinear form $\Qtl \colon \HQtl \tensor \HQtl \to
\QQ(-\ell_1)$. If we define the integral structure $\HZtl$ by taking the image of
$\HZ$, the induced period mapping
\[
	\Phitl(z) = e^{\sum_{j \not\in J_1} z_j \Ntl_j} e^{\Gammatl(s)} \Ftl
\]
gives a variation of $\ZZ$-Hodge structure of weight $\ell_1$ on $\HCtl$,
polarized by $\Qtl$. 
\end{proof}

We note that this construction reduces the value of $d$, in the following sense.

\begin{corollary} \label{cor:induction}
Notation being as above, $\Phitl \bigl( z(m) \bigr)$ only depends on 
\[
	\ztl(m) = i A \ttl(m) + w(m),
\]
where $\ttl(m) = \bigl( 0, t_2(m), \dotsc, t_d(m) \bigr)$, and $A$ and $w(m)$ are as
in \eqref{eq:z-t}.
\end{corollary}

\begin{proof}
Because $\Phi(z)$ is a nilpotent orbit in the variables $\{s_j\}_{j \in J_1}$, it is
clear that $\Phitl(z)$ only depends on the variables $\{z_j\}_{j \not\in J_1}$; but
$z_j(m) = \ztl_j(m)$ for $j \not\in J_1$.
\end{proof}

Now suppose that $\ell_1 = 0$, so that we are again dealing with a polarized
variation of $\ZZ$-Hodge structure of weight zero. Suppose we have a sequence $h(m)
\in \HR$ with
\[
	h(m) \equiv b(m) \mod \Phi^0 \bigl( z(m) \bigr)
\]
for a harmless sequence $b(m) \in \HC$. Let $\htl(m) \in \HRtl$ denote the image of
$h(m)$; note that $\htl(m) \in \HZtl$ if the initial sequence satisfies $h(m) \in \HZ$.
\lemmaref{lem:harmless-W} allows us to assume that the harmless sequence $b(m)$ lies
in $W_{\ell_1}^1$; consequently, 
\[
	\htl(m) \equiv \btl(m) \mod \Phitl^0 \bigl( z(m) \bigr)
\]
for a harmless sequence $\btl(m) \in \HCtl$. The sequence $\htl(m)$ automatically 
inherits the following boundedness property from $h(m)$.

\begin{lemma} 
If $Q \bigl( h(m), h(m) \bigr)$ is bounded, then $\Qtl \bigl( \htl(m), \htl(m)
\bigr)$ is also bounded.
\end{lemma}

\begin{proof}
By \propositionref{prop:bounded}, the assertion is equivalent to the boundedness of
the sequence $\etl(m) \htl(m)$; note that this requires $\ell_1 = 0$. But clearly
\[
	\norm{\etl(m) \htl(m)} \leq t_1(m)^{-\ell_1/2} \norm{e(m) h(m)},
\]
which is bounded as long as $\ell_1 \geq 0$. 
\end{proof}

\subsection{Position relative to the $\ZZ^d$-grading}
\label{par:position}

Here we present a streamlined version of \cite[Lemma~4.4 and Lemma~4.5]{CDK}. We
relax the condition on the weight, and only assume that $\Phi(z)$ is the period
mapping of a polarized variation of $\ZZ$-Hodge structure of weight $w \geq 0$. We
also fix a sequence $z(m) = i A t(m) + w(m)$ as in \eqref{eq:z-t}, and consider on
$\HC$ the $\ZZ^d$-grading defined by $Y_1, \dotsc, Y_d$.

\begin{definition}
The \define{position} of a sequence $h(m) \in \HR$ relative to the $\ZZ^d$-grading is
the largest multi-index $(\ell_1, \dotsc, \ell_d) \in \ZZ^d$ (in the lexicographic
ordering) with the property that $h(m)^{(\ell_1, \dotsc, \ell_d)} \neq 0$ for
infinitely many $m \in \NN$.
\end{definition}

Now suppose we are given a sequence $h(m) \in \HR$ that is in the position $(\ell_1,
\dotsc, \ell_d)$ relative to the $\ZZ^d$-grading defined by $Y_1, \dotsc, Y_d$.
Assume moreover that 
\[
	\norm{h(m)^{(\ell_1, \dotsc, \ell_d)}} \geq \eps
\]
for a positive constant $\eps > 0$; this replaces the condition that $h(m) \in \HZ$.
Our goal is to show that if $t_1(m)^{w/2} \cdot e(m) h(m)$ is bounded, then
$h(m)$ must be in the position $(-w, 0, \dotsc, 0)$ relative to the $\ZZ^d$-grading.

\begin{proposition} \label{prop:weights}
Suppose that we have $h(m) \equiv b(m) \mod \Phi^0 \bigl( z(m) \bigr)$ for a sequence 
$b(m) \in \HC$ with $\norm{b(m)}$ in $O(e^{-\alpha t_1(m)})$.
If $t_1(m)^w \norm{e(m) h(m)}^2$ is bounded, then $w + \ell_1 = \ell_2 = \dotsb =
\ell_d = 0$.
\end{proposition}

\begin{proof}
The key observation is that the ratios
\[
	\frac{\norm{e(m) b(m)}^2}{\norm{e(m) h(m)}^2}
\]
are going to zero. Indeed, $\norm{e(m) b(m)}^2$ is in $O(e^{-2\alpha t_1(m)})$,
whereas $\norm{e(m) h(m)}^2$ is bounded from below by
\[
	t_1(m)^{\ell_1} \dotsm t_d(m)^{\ell_d} \norm{h(m)^{(\ell_1, \dotsc, \ell_d)}}^2
		\geq \eps^2 \cdot t_1(m)^{\ell_1} \dotsm t_d(m)^{\ell_d}.
\]
The unit vectors $\norm{e(m) h(m)}^{-1} \cdot e(m) h(m)$ therefore converge to a unit
vector in $W_{w + \ell_1}^1 \cap \Fsh^0 \cap \HR$, and so $w + \ell_1 \geq 0$ because
$(W^1, \Fsh)$ is a mixed Hodge structure.

Because of the bound on $e(m) h(m)$, we know that $\norm{h(m)}$ grows at most like a
power of $t_1(m)$. We can therefore assume that $\Phi(z)$ is a nilpotent orbit in the
variables $\{s_j\}_{j \in J_1}$ (by \lemmaref{lem:nilpotent-J} below), and that $b(m)
\in W_{\ell_1}^1$ (by \lemmaref{lem:harmless-W}). We now project the sequence to
$\HCtl = \gr_{w + \ell_1}^{W^1}$, which carries a polarized variation of Hodge
structure of weight $w + \ell_1$ by \propositionref{prop:induction}. The new sequence
$\htl(m)$ is in the position $(\ell_2, \dotsc, \ell_d)$ relative to the
$\ZZ^{d-1}$-grading on $\HCtl$, and exponentially close to $\Phitl^0 \bigl( z(m)
\bigr)$.  Moreover, the expression
\[
	t_2(m)^{w+\ell_1} \norm{\etl(m) \htl(m)}^2 
	\leq t_1(m)^{w+\ell_1} \norm{\etl(m) \htl(m)}^2
	\leq t_1(m)^w \norm{e(m) h(m)}^2
\]
is bounded (because $w + \ell_1 \geq 0$), and we still have
\[
	\norm{\htl(m)^{(\ell_2, \dotsc, \ell_d)}} 
		= \norm{h(m)^{(\ell_1, \dotsc, \ell_d)}}
		\geq \eps.
\]
By induction, $w + \ell_1 + \ell_2 = \ell_3 = \dotsb = \ell_d = 0$. But now we get
\[
	t_1(m)^w \norm{e(m) h(m)}^2 \geq 
		\eps^2 t_1(m)^{w+\ell_1} t_2(m)^{\ell_2} \dotsm t_d(m)^{\ell_n}
		= \eps^2 \left( \frac{t_1(m)}{t_2(m)} \right)^{w+\ell_1}.
\]
This can only be bounded if $w + \ell_1 = 0$, and hence $\ell_2 = 0$.
\end{proof}

The following lemma was used during the proof; it will make another appearance when
we prove \theoremref{thm:local3}. We put the period mapping into the standard
form $\Phi(z) = e^{\sum z_j N_j} e^{\Gamma(s)} F$. Let $\Gamma_1(s)$ denote the
result of setting all the variables $\{s_j\}_{j \in J_1}$ in $\Gamma(s)$ to zero, and
define $\Phi_1(z) = e^{\sum z_j N_j} e^{\Gamma_1(s)} F$, which is now a nilpotent
orbit in the variables $\{s_j\}_{j \in J_1}$.

\begin{lemma} \label{lem:nilpotent-J}
Suppose that $h(m)$ is congruent, modulo $\Phi^0 \bigl( z(m) \bigr)$, to a
sequence that is harmless with respect to $t(m)$. If $\norm{h(m)}$ is in
$O(t_1(m)^N)$ for some $N \in \NN$, then the same is true modulo
$\Phi_1^0 \bigl( z(m) \bigr)$.
\end{lemma}

\begin{proof}
Suppose that $h(m) \equiv b(m) \mod \Phi^0 \bigl( z(m) \bigr)$. We have
\[
	e^{\sum z_j N_j} e^{\Gamma(s)} = 
		\left( e^{\sum z_j N_j} e^{\Gamma(s)} e^{-\Gamma_1(s)} e^{-\sum z_j N_j} \right)
		\cdot e^{\sum z_j N_j} e^{\Gamma_1(s)},
\]
and because $\abs{s_j(m)}$ is in $O(e^{-\alpha t_1(m)})$, \propositionref{prop:Gamma}
shows that the difference 
\[
	\Delta(m) = 
		\left( e^{\sum z_j(m) N_j} e^{\Gamma_1(s(m))} e^{-\Gamma(s(m))} 
		e^{-\sum z_j(m) N_j} \right) - \id
\]
is an operator whose norm is in $O(e^{-\alpha t_1(m)})$. We therefore obtain
\[
	h(m) \equiv b(m) + \Delta(m) \bigl( b(m) - h(m) \bigr) 
		\mod e^{\sum z_j(m) N_j} e^{\Gamma_1(s(m))} F^0,
\]
and because $\norm{b(m)}$ is bounded and $\norm{h(m)}$ is in $O(t_1(m)^N)$, the
sequence on the right-hand side is still harmless with respect to $t(m)$. 
\end{proof}

\subsection{Proof in the general case}
\label{par:proof-n}

We now prove \theoremref{thm:local3} by induction on $d \geq 1$. As in the
one-variable case, the argument can be divided into six steps. 

\paragraph{Step 1}

To get started, we have to prove that the sequence $h(m) \in \HZ(K)$ is bounded in
the Hodge norm at the point $\Phi \bigl( z(m) \bigr)$. This follows immediately from
\propositionref{prop:bounded}. As in the one-variable case, we will later use only
the equivalent fact that the sequence $e(m) h(m)$ is bounded. We also note that the
sequence $\norm{h(m)}$ grows at most like a power of $t_1(m)$.

\paragraph{Step 2}

We now reduce to the case where $\Phi(z)$ is a nilpotent orbit in the variables
$\{s_j\}_{j \in J_1}$; those are the ones that are going to zero most
quickly. Recall that
\[
	\Phi(z) = e^{\sum z_j N_j} e^{\Gamma(s)} F;
\]
let $\Gamma_1(s)$ denote the result of setting $s_j = 0$ for every $j \in J_1$. The
claim is that we can replace $\Gamma(s)$ by $\Gamma_1(s)$ without affecting any of
the conditions of the problem; this is proved in \lemmaref{lem:nilpotent-J}.
After making the obvious replacements, we can therefore assume without loss of
generality that the operator $\Gamma(s)$ does not depend on the variables $s_j$ with 
with $j \in J_1$. In particular, $e^{\Gamma(s)}$ now commutes with $T_1$ by
\propositionref{prop:Gamma}, and therefore preserves the weight filtration $W^1$.
Note that we only have $\Phi(z) \in D$ when all the imaginary parts of $z \in \HH^n$
are sufficiently large; after passing to a subsequence, we may assume that this is
the case along our sequence $z(m)$.

\paragraph{Step 3}

Our next goal is to show that $h(m) \in W_0^1$. As in the one-variable case, we will
deduce this from the boundedness of the sequence $e(m) h(m) \in \HR$. Let $\ell \in
\ZZ^d$ be the largest index (in the lexicographic ordering) with the
property that $h(m)^{(\ell_1, \dotsc, \ell_d)}$ is nonzero for infinitely many $m$.
After passing to a subsequence, we therefore have $h(m) \in W_{\ell_1}^1$; its
projection to $\gr_{\ell_1}^{W^1}$ lies in the image of $W_{\ell_1+\ell_2}^2$, and so
on. Note that the projection 
\[
	\HC^{(\ell_1, \dotsc, \ell_d)} \to 
		\gr_{\ell_1 + \dotsb + \ell_d}^{W^d} \dotsb 
			\gr_{\ell_1 + \ell_2}^{W^2} \gr_{\ell_1}^{W^1}
\]
is an isomorphism; because $h(m) \in \HZ$, it follows that $h(m)^{(\ell_1, \dotsc,
\ell_d)}$ takes values in a discrete set.  In particular, we have
$\norm{h(m)^{(\ell_1, \dotsc, \ell_d)}} \geq \eps$ for a constant $\eps > 0$. 

Now suppose that $h(m) \not\in W_0^1$; in other words, suppose that $\ell_1 \geq 1$.
Define $\HCtl = \gr_{\ell_1}^{W^1}$; according to \propositionref{prop:induction}, it
again supports a polarized variation of $\ZZ$-Hodge structure of weight $\ell_1$. Let
$\htl(m)$ denote the image of $h(m)$ in $\HCtl$. Because $\Phi(z)$ is a nilpotent
orbit in the variables $\{s_j\}_{j \in J_1}$, we can use \lemmaref{lem:harmless-W} to
make sure that $b(m) \in W_{\ell_1}^1$. In the congruence
\[
	h(m) - \bigl( b_1(m) + \dotsb + b_d(m) \bigr) \equiv b_0(m) 
		\mod \Phi^0 \bigl( z(m) \bigr),
\]
the term in parentheses is contained in $\ker T_1 \subseteq W_0^1$, and therefore
disappears when we project to $\HCtl$. Under the assumption that $\ell_1 \geq 1$, our
sequence $\htl(m)$ is therefore exponentially close to the subspace $\Phitl^0
\bigl( z(m) \bigr)$. We can now apply \propositionref{prop:weights} to the sequence
$\htl(m)$ and the polarized variation of Hodge structure $\Phitl(z)$ on $\HCtl$; the
result is that $\ell_1 + \ell_2 = 0$ and $\ell_3 = \dotsb = \ell_d = 0$. But then
\[
	\norm{e(m) h(m)}^2 \geq \norm{e(m) h(m)^{(\ell_1, \dotsc, \ell_d)}}^2 
		\geq  \eps^2 \left( \frac{t_1(m)}{t_2(m)} \right)^{\ell_1},
\]
and since $\ell_1 \geq 1$, this inequality contradicts the boundedness of $e(m)
h(m)$. Consequently, $h(m) \in W_0^1$ after all.

\paragraph{Step 4}

Using the notation from \parref{par:induction}, we now apply the induction hypothesis
to the sequence $\bigl( \ztl(m), \htl(m) \bigr)$ and the period mapping $\Phitl(z)$
on the space $\HCtl = \gr_0^{W^1}$; the construction in
\propositionref{prop:induction} shows that all the assumptions are again satisfied,
but with a smaller value of $d$. After passing to a subsequence, $\htl(m)$ has a
constant value $\htl \in \HZtl$, and $\Ttl_k \htl = 0$ for $k = 2, \dotsc, d$. In
order to lift these results back to $\HC$, we define
\[
	h_0 = \sum_{\ell_2, \dotsc, \ell_d} h(m)^{(0, \ell_2, \dotsc, \ell_d)} \in \HR;
\]
note that $h_0$ is constant, because it projects to the constant sequence $\htl$
under the isomorphism $E_0(Y_1) \simeq \HCtl$. We also have $h_0 \in W_0^k$ for every
$k = 2, \dotsc, d$, because $\htl \in \Wtl_0^k$. The conclusion is that
\[
	h(m) \equiv h_0 \mod W_{-1}^1.
\]
Our next task is to prove that $T_1 h_0^{(0, \dotsc, 0)} = 0$. 

If we apply the operator $e(m)$ to the congruence in \eqref{eq:assumption}, we obtain
\[
	e(m) h(m) \equiv e(m) b(m) \mod e(m) \Phi^0 \bigl( z(m) \bigr).
\]	
Because $e(m) h(m)$ is bounded, and because we have already computed the limit of $e(m)
b(m)$ in  \lemmaref{lem:limit-bm}, we can pass to a subsequence where
\[
	v = \lim_{m \to \infty} e(m) h(m) \equiv 
		\lim_{m \to \infty} b_d(m)^{(0, \dotsc, 0)} \mod \Fsh^0.
\]
Now comes the crucial point: by \lemmaref{lem:limit-bm}, the right-hand side of the
congruence is an element of $E_0(Y_k) \cap \ker T_k$ for every $k = 1, \dotsc, d$.
Because $v \in W_0^1 \cap \HR$, we can apply \lemmaref{lem:MHS} from the one-variable
case to the $\RR$-split mixed Hodge structure $\bigl( W^1, e^{-i T_1} \Fsh \bigr)$
and conclude that $v \in E_0(Y_1)$ and $T_1 v = 0$.

On the other hand, we can project the congruence
\[
	e(m) h(m) \equiv e(m) h_0 \mod W_{-1}^1
\]
to the subspace $E_0(Y_1)$; because $h_0 \in W_0^k$ for every $k = 1, \dotsc, d$, we
get
\[
	v = \lim_{m \to \infty} e(m) h_0 = h_0^{(0, \dotsc, 0)}.
\]
In particular, we have $T_1 h_0^{(0, \dotsc, 0)} = 0$.

\paragraph{Step 5}

Now we show that $\norm{T_1 h(m)}$ is in $O(e^{-\alpha t_1(m)})$; the method is
almost the same as in the one-variable case. We have
\[
	e(m) T_1 h(m) \equiv e(m) T_1 b_0(m) 
		\mod T_1 e(m) \Phi^0 \bigl( z(m) \bigr);
\]
here we used the fact that $T_1 e(m) = t_1(m) \cdot e(m) T_1$, because $T_1 = \Th_1$
commutes with $Y_2, \dotsc, Y_d$ and satisfies $\lie{Y_1}{T_1} = - 2 T_1$. We
claim that $\norm{e(m) T_1 h(m)}$ is bounded by a constant multiple of $\norm{e(m)
T_1 b_0(m)}$. If not, then the ratios
\[
	\frac{\norm{e(m) T_1 b_0(m)}}{\norm{e(m) T_1 h(m)}}
\]
are going to zero. After passing to a subsequence, the sequence of unit vectors
\begin{equation} \label{eq:unit}
	u(m) = \frac{e(m) T_1 h(m)}{\norm{e(m) T_1 h(m)}} \in W_{-2}^1 \cap \HR 
\end{equation}
converges to a unit vector $u \in W_{-2}^1 \cap \Fsh^{-1} \cap \HR$. Now $\bigl( W^1,
e^{-i T_1} \Fsh \bigr)$ is an $\RR$-split mixed Hodge structure; we can therefore
apply \lemmaref{lem:MHS} from the one-variable case to deduce that $u \in
E_{-2}(Y_1)$.

Recall that the decomposition $W_{-2}^1 = E_{-2}(Y_1) \oplus W_{-3}^1$ is orthogonal
with respect to the inner product on $\HC$. If we project the congruence
\[
	u(m) \equiv \frac{e(m) T_1 h_0}{\norm{e(m) T_1 h(m)}} \mod W_{-3}^1
\]
to the subspace $E_{-2}(Y_1)$, we find that 
\[
	u = \lim_{m \to \infty} \frac{e(m) T_1 h_0}{\norm{e(m) T_1 h(m)}}.
\]
Because the right-hand side belongs to $W_{-2}^2 \cap \dotsb \cap W_{-2}^d$, it
follows that $u$ lies in the intersection $W_{-2}^1 \cap \dotsb \cap W_{-2}^d \cap
\Fsh^{-1} \cap \HR$. We can therefore apply \lemmaref{lem:MHS} again, to the
$\RR$-split mixed Hodge structure 
\[
	\bigl( W^k, e^{-i(\Th_1 + \dotsb + \Th_k)} \Fsh \bigr),
\]
to show that $(Y_1 + \dotsb + Y_k) u = -2 u$ for every $k = 1, \dotsb, d$. These
relations are saying that $u \in \HC^{(-2, 0, \dotsc, 0)}$. But if we project
\eqref{eq:unit} to that summand and use the fact that $h(m) \equiv h_0 \mod
W_{-1}^1$, we find that
\[
	u(m)^{(-2, 0, \dotsc, 0)} 
		= \frac{e(m) T_1 h_0^{(0, \dotsc, 0)}}{\norm{e(m) T_1 h(m)}}
		= 0.
\]
This forces $u = 0$, in contradiction to the fact that $u$ is a
unit vector. Consequently, $\norm{e(m) T_1 h(m)}$ must be bounded by a constant
multiple of $\norm{e(m) T_1 b_0(m)}$, and therefore in $O(e^{-\alpha
t_1(m)})$. Because $e(m)^{-1}$ grows at most like a power of $t_1(m)$, this is enough
to conclude that $\norm{T_1 h(m)}$ is exponentially small.

\paragraph{Step 6}

We can now complete the proof by the method of \cite[4.9]{CDK}. If $d \geq 2$, we
observe that
\begin{equation} \label{eq:induction}
\begin{split}
	e^{-i t_1(m) T_1} h(m)
		\equiv e^{-i t_1(m) T_1} b_0(m) + & b_1(m) + \dotsb + b_{d}(m) \\
			&\mod e^{-i t_1(m) T_1} \Phi^0 \bigl( z(m) \bigr).
\end{split}
\end{equation}
Because $\norm{T_1 h(m)}$ is in $O(e^{-\alpha t_1(m)})$, it follows that $h(m)$ is
the sum of a harmless element and an element of $e^{-i t_1(m) T_1} \Phi^0 \bigl( z(m)
\bigr)$. Remembering that $\Phi(z)$ is a nilpotent orbit in the variables $\{s_j\}_{j
\in J_1}$, the sequence of filtrations
\[
	e^{-i t_1(m) T_1} \Phi \bigl( z(m) \bigr) \in D
\]
no longer involves either $t_1(m)$ or $T_1$; this means that we have managed to
reduce the value of $d$. By induction, we can pass to a subsequence and arrange that
$h(m)$ is constant and in the kernel of $T_2, \dotsc, T_d$. Since $T_1 h(m)$ is
exponentially small, it has to be zero as well, concluding the proof in the case $d
\geq 2$.

If $d = 1$, then we argue as in the one-variable case. Recall that
\[
	\Phi \bigl( z(m) \bigr) = e^{\sum z_j(m) N_j} F
		 = e^{i t(m) T_1} e^{\sum w_j(m) N_j} F
\]
is a nilpotent orbit, with $w(m) \in \CC^n$ convergent and $\Phi \bigl( w(m) \bigr)
\in D$. The formula in \eqref{eq:induction} shows that the Hodge norm of $h(m)$ with
respect to $\Phi \bigl( w(m) \bigr)$ is bounded. Since these Hodge filtrations lie in
a compact set, $\norm{h(m)}$ must be bounded; after passing to a subsequence, $h(m)$
is constant, and then $T_1 h(m) = 0$ as before.

\section{Construction of the extension space}

\subsection{Setup and basic properties}
\label{par:setup}

Let $X$ be a complex manifold, $Z \subseteq X$ an analytic subset, and $\shH$ a
polarized variation of $\ZZ$-Hodge structure on $X_0 = X \setminus Z$. We denote by
$\shHZ$ the underlying local system of free $\ZZ$-modules, and by $Q \colon \shHQ
\tensor_{\QQ} \shHQ \to \QQ(0)$ the bilinear form giving the polarization. 

Now let $\TZ$ be the \'etal\'e space of the local system $\shHZ$; it is a (usually
disconnected) covering space of the complex manifold $X_0$. Point of $\TZ$ may be
thought of as pairs $(x, h)$, where $x \in X_0$ and $h \in \shH_{\ZZ, x}$ is a class
in the stalk. As in the introduction, we define
\[
	\TZ(K) = \menge{(x,h) \in \TZ}{\abs{Q_x(h,h)} \leq K}
\]
for every $K \geq 0$; note that it is a union of connected components of $\TZ$,
because the function $(x,h) \mapsto Q_x(h,h)$ is constant on each connected
component. Let $T(\shH) = \Spec(\Sym \shHd)$ be the vector bundle with sheaf of
sections $\shHd$; we similarly define $T(F^1 \shH)$.

We first describe in more detail how the holomorphic mapping $\eps \colon \TZ \to T(F^1
\shH)$ is constructed. The pairing $Q$ induces an injective morphism of sheaves
\[
	\shHZ \into \shHd, \quad h \mapsto Q(h, \argbl);
\]
it is injective because $Q$ is nondegenerate. As in \cite[Section~2.6]{Schnell-N},
this morphism gives rise to a holomorphic mapping
\[
	\TZ \into T(\shH),
\]
which embeds the complex manifold $\TZ$ into the holomorphic vector bundle $T(\shH)$.
From now on, we identify $\TZ$ with a complex submanifold of $T(\shH)$.  We obtain
$\eps \colon \TZ \to T(F^1 \shH)$ by composing with the projection $q \colon T(\shH)
\to T(F^1 \shH)$. 

Now fix some $K \geq 0$. We already know from the result about Hodge structures in
\lemmaref{lem:fiber} that $\eps \colon \TZ(K) \to T(F^1 \shH)$ has finite fibers; the
purpose of this section is to understand its global properties. The
following diagram shows all the relevant mappings:
\[
\begin{tikzcd}[column sep=large]
& T(\shH) \dar{q} \\
\TZ(K) \arrow{dr}{\pi} \rar{\eps} \arrow[hook]{ur} & T(F^1 \shH) \dar \\
& X_0.
\end{tikzcd}
\]
The polarization defines a hermitian metric on the holomorphic vector bundle
associated with $\shH$, the so-called \define{Hodge metric}. It induces hermitian metrics on
the two bundles $T(\shH)$ and $T(F^1 \shH)$. Let $B_r(\shH) \subseteq T(\shH)$ denote
the closed tube of radius $r > 0$ around the zero section. The proof of
\lemmaref{lem:fiber} shows that 
\[
	\eps^{-1} \bigl( B_r(\shH) \bigr) 
		\subseteq B_{\sqrt{K + 4 r^2}} \bigl( F^1 \shH \bigr);
\]
in particular, the general discussion in \parref{par:covering} applies to our
situation. We summarize the results in the following proposition.

\begin{proposition} \label{prop:eps}
The holomorphic mapping $\eps \colon \TZ(K) \to T(F^1 \shH)$ is finite,
and its image is a closed analytic subset of $T(F^1 \shH)$. Moreover, the induced mapping
from $\TZ(K)$ to the normalization of the image is a finite covering space.
\end{proposition}

\begin{proof}
This is proved in \parref{par:covering} below.
\end{proof}

\subsection{Analyticity of the closure}

In this section, we prove \theoremref{thm:main} in general.  We denote by $M$ the
polarized Hodge module of weight $\dim X$ with strict support $X$, canonically
associated with $\shH$ by the equivalence of categories in
\cite[Theorem~3.21]{Saito-MHM}. Let $(\Mmod, F_{\bullet} \Mmod)$ denote the
underlying filtered regular holonomic $\Dmod_X$-module. By construction, the
restriction of $F_{-1} \Mmod$ to the open subset $X_0$ is isomorphic to $F^1 \shH$.
The analytic space $T(F_{-1} \Mmod)$ therefore contains an open subset isomorphic to
the vector bundle $T(F^1 \shH)$. We denote by
\[
	\eps \colon \TZ(K) \to T(F_{-1} \Mmod)
\]
the resulting holomorphic mapping.

\begin{theorem} \label{thm:closure}
The closure of the image of the holomorphic mapping 
\[
	\eps \colon \TZ(K) \to T(F_{-1} \Mmod)
\]
is an analytic subset of $T(F_{-1} \Mmod)$.
\end{theorem}

\begin{proof}
There is a proper holomorphic mapping $f \colon Y \to X$, whose restriction to $Y_0 =
f^{-1}(X_0)$ is a finite covering space, such that $D = f^{-1}(Z)$ is a divisor with normal
crossings, and such that the local monodromy of $\fu_0 \shH$ at every point of $D$ is
unipotent. To construct $f$, we first take an embedded resolution of singularities of
$(X, Z)$. According to \cite[Lemma~4.5]{Schmid}, the pullback of $\shH$ has
quasi-unipotent local monodromy at every point of the preimage of $Z$; after a finite
branched covering and a further resolution of singularities, we arrive at the stated
situation.

Now let $M'$ denote the polarized Hodge module of weight $\dim Y$ with strict
support $Y$, associated with $\shH' = \fu_0 \shH$. According to
\cite[Lemma~2.21]{Schnell-N}, there is a canonical morphism
\[
	F_{-1} \Mmod' \to \fu F_{-1} \Mmod,
\]
whose restriction to $Y_0$ is an isomorphism. We then have the following commutative
diagram of holomorphic mappings:
\[
\begin{tikzcd}
\TZ(K) \dar{\eps} & \TZ(K) \times_X Y \lar[swap]{p_1} 
	\dar{\eps \times \id} \arrow[equal]{r} & \TZ'(K) \dar{\eps'} \\
T(F_{-1} \Mmod) & T(F_{-1} \Mmod) \times_X Y \lar[swap]{p_1} \rar{g} & T(F_{-1} \Mmod')
\end{tikzcd}
\]
By \theoremref{thm:main-local}, the closure of the image of $\eps'$ is analytic.
The same is therefore true for $\eps \times \id$, because $g$ is an isomorphism over
$Y_0$. Because $f$ is proper, the result for $\eps$ now follows from Remmert's
proper mapping theorem \cite[III.4.3]{SCV7}.
\end{proof}

\subsection{Extension of the finite mapping}

We are now ready to prove the main result, namely that $\eps \colon \TZ(K) \to
T(F_{-1} \Mmod)$ can be extended to a finite mapping.

\begin{theorem} \label{thm:extension}
There is a normal analytic space $\TZt(K)$ containing the complex manifold $\TZ(K)$
as a dense open subset, and a finite holomorphic mapping 
\[
	\epst \colon \TZt(K) \to T(F_{-1} \Mmod),
\]
whose restriction to $\TZ(K)$ agrees with $\eps$. Moreover, $\TZt(K)$ and $\epst$ are
unique up to isomorphism.
\end{theorem}

\begin{proof}
The closure of the image of $\eps$ is an analytic subset of $T(F_{-1} \Mmod)$
according to \theoremref{thm:closure}. Let $W$ denote its normalization; according to
\propositionref{prop:eps}, the induced mapping from $\TZ(K)$ to $W$ is a finite
covering space over its image. The \emph{Fortsetzungssatz} of Grauert and Remmert
\cite[VI.3.3]{SCV7} shows that it extends in a unique way to a finite branched
covering of $W$. If we define $\TZt(K)$ to be the analytic space in this covering,
and $\epst \colon \TZt(K) \to T(F_{-1} \Mmod)$ to be the induced holomorphic mapping,
then all the requirements are fulfilled. The last assertion follows from the
uniqueness statement in \cite[VI.3.3]{SCV7}.
\end{proof}

As a consequence, we obtain a canonical analytic space that contains the locus of
Hodge classes $\Hdg(\shH) \cap \TZ(K)$ and is finite over $X$.

\begin{corollary} \label{cor:Hdg}
The locus of Hodge classes $\Hdg(\shH) \cap \TZ(K)$ extends in a canonical way to an 
analytic space that is finite over $X$.
\end{corollary}

\begin{proof}
We can take the preimage of the zero section in $T(F_{-1} \Mmod)$ under
the finite holomorphic mapping $\epst \colon \TZt(K) \to T(F_{-1} \Mmod)$.
\end{proof}

\subsection{General results about certain covering spaces}
\label{par:covering}

In this section, we consider the following general situation. Let $X$ be a complex manifold,
and suppose that we have a surjective mapping $q \colon E_1 \to E_2$ between two
holomorphic vector bundles on $X$. We assume that $E_1$ has a hermitian metric $h_1$,
and we endow $E_2$ with the induced hermitian metric $h_2$. Lastly, we shall assume
that we have a complex submanifold $T \into E_1$, with the property that $\pi
\colon T \to X$ is a (possibly disconnected) covering space. We denote by $\eps
\colon T \to E_2$ the induced holomorphic mapping; see also the diagram below.
\[
\begin{tikzcd}[column sep=large]
& E_1 \dar{q} \arrow[bend
left=40]{dd}{p_1} \\
T \arrow{dr}{\pi} \rar{\eps} \arrow[hook]{ur} & E_2 \dar{p_2} \\
& X.
\end{tikzcd}
\]
For any real number $r > 0$, we denote by $B_r(E_j)$ the closed tube of radius $r$ around
the zero section in the vector bundle $E_j$. We assume the following condition:
\begin{equation} \label{eq:condition}
	\text{For every $r > 0$, there exists $R > 0$ with
		$T \cap \eps^{-1} \bigl( B_r(E_2) \bigr) \subseteq B_R(E_1)$.}
\end{equation}

\begin{lemma}
If \eqref{eq:condition} holds, then $\eps \colon T \to E_2$ is a finite mapping.
\end{lemma}

\begin{proof}
Recall that a holomorphic mapping is called \define{finite} if it is closed and has
finite fibers \cite[I.2.4]{SCV7}; an equivalent condition is that the mapping is
proper and has finite fibers. Let us first show that $\eps$ is proper. Given an
arbitrary compact subset $K \subseteq E_2$, we can find $r > 0$ such that $K
\subseteq B_r(E_2)$. According to \eqref{eq:condition}, the preimage $\eps^{-1}(K)$
is contained in $B_R(E_1)$ for some $R > 0$; because it is closed, it must be
compact. Now it is easy to show that $\eps$ has finite fibers: the fibers of $\eps$
are contained in the fibers of $\pi$, which are discrete because $\pi \colon T
\to X$ is a covering space; being compact, they must therefore be finite sets.
\end{proof}

\begin{corollary}
The image of $\eps$ is an analytic subset of $E_2$.
\end{corollary}

\begin{proof}
This follows from the finite mapping theorem \cite[I.8.2]{SCV7}, which is a special
case of Remmert's proper mapping theorem.
\end{proof}

Of course, $\eps$ is still a local biholomorphism; but the images of different sheets
of the covering space $T$ may intersect in $E_2$. This picture suggests the following
result about the normalization of $\eps(T)$.

\begin{lemma}
The normalization of $\eps(T)$ is a complex manifold, and the induced mapping from
$T$ to the normalization is a finite covering space.
\end{lemma}

\begin{proof}
Let $Y$ denote the normalization of $\eps(T)$; for the construction, see
\cite[I.14.9]{SCV7}. Because $T$ is a complex manifold, we obtain a factorization
\[
\begin{tikzcd}
T \rar{f} \arrow[bend left=40]{rr}{\eps} \arrow{drr}{\pi} 
		& Y \rar{\nu} & E_2 \dar{p_2} \\
	& & X;
\end{tikzcd}
\]  
note that $f$ is again a finite mapping. According to \cite[I.13.1]{SCV7}, $Y$ is
locally irreducible; now \cite[I.10.14]{SCV7} implies that $f \colon T \to Y$ is
open. Since $\pi \colon T \to X$ is a covering space, this is enough to guarantee
that $Y$ is again a complex manifold, and that $f \colon T \to Y$ is a finite
covering space.
\end{proof}

\section{The universal family of hyperplane sections}

\subsection{Description of the variation of Hodge structure}

The purpose of this chapter is to apply the general construction from above to the
universal family of hyperplane sections of a smooth projective variety. Let $X$ be a
smooth projective variety of odd dimension $2n+1$, and let $L$ be a very ample line
bundle on $X$. It determines an embedding of $X$ into the projective space $\PP \bigl(
H^0(X, L) \bigr)$. We denote by $B = \PP \bigl( H^0(X, L)^{\ast} \bigr)$ the dual
projective space; a point $b \in B$ corresponds to a hyperplane $H_b$, and therefore
to a hyperplane section $H_b \cap X$ of $X$. There is a natural incidence variety
\[
	\famX = \menge{(b,x) \in B \times X}{x \in H_b \cap X};
\]
it is a projective bundle over $X$, and therefore again
a smooth projective variety of dimension $2n + \dim B$. Let $f \colon \famX \to B$
denote the first projection, and $f_0 \colon \famX_0 \to B_0$ its restriction to the
Zariski-open subset where $H_b \cap X$ is nonsingular. 

On $B_0$, we have a polarized variation of $\ZZ$-Hodge structure $\shH$ of weight
zero, obtained by taking the quotient of $R^{2n} f_{0\ast} \ZZ(n)$ by the constant
part $H^{2n} \bigl( X, \ZZ(n) \bigr)$; note that the polarization is only defined
over $\QQ$ in general. Recall that for a smooth hyperplane section $Y = H \cap X$,
the quotient
\[
	H^{2n} \bigl( Y, \ZZ(n) \bigr) \big/ H^{2n} \bigl( X, \ZZ(n) \bigr) 
\]
is torsion-free (by the Lefschetz theorems); tensored with $\QQ$, it becomes
isomorphic to the variable part
\[
	\ker \Bigl( H^{2n} \bigl( Y, \QQ(n) \bigr) \to 
		H^{2n+2} \bigl( X, \QQ(n+1) \bigr) \Bigr),
\]
and therefore canonically polarized by the intersection product on $Y$.

As usual, let $M$ denote the polarized Hodge module of weight $\dim B$ with strict
support $B$, associated with $\shH$. In this situation, the filtered $\Dmod$-module
$(\Mmod, F_{\bullet})$ can be described concretely in terms of residues
\cite{Schnell-R}. Recall that when $Y = H \cap X$ is a smooth hyperplane section, we have
a residue mapping
\[
	\Res_Y \colon H^0 \bigl( X, \OmX^{2n+1}(kY) \bigr) \to F^{2n+1-k} H^{2n}(Y, \CC).
\]	
By applying this construction on each smooth hyperplane section, we can obtain
sections of $\shH$ from meromorphic $(2n+1)$-forms on $B \times X$ with poles along
$\famX$. To state the precise result, let $j \colon B_0 \into B$ denote the
inclusion. Then $\Mmod$ is a subsheaf of $\jl \shH$, and the space of sections $H^0(U,
F_k \Mmod)$ on an open set $U \subseteq B$ consists exactly of those $s \in H^0(U,
\jl \shH)$ that satisfy
\[
	s(b) = \Res_{H_b \cap X} \Bigl( \omega \restr{\{b\} \times X} \Bigr) \quad
		\text{at every point $b \in U \cap B_0$}
\]
for some choice of meromorphic $(2n+1)$-form
\[
	\omega \in H^0 \Bigl( U \times X, \Omega_{B \times X}^{2n+1} 
		\bigl( (n+1+k) \famX \bigr) \Bigr).
\]
In addition to this description, the following result is proved in
\cite[Corollary~4]{Schnell-R}.

\begin{theorem} \label{thm:ample}
The coherent sheaf $F_k \Mmod$ is a quotient of the ample vector bundle 
\[
	H^0 \bigl( X, \OmX^{2n+1} \tensor L^{n+1+k} \bigr) \tensor \shO_B(n+1+k), 
\]
and therefore globally generated.
\end{theorem}

\subsection{Properties of the extension space}

Now let us see what our general construction produces in the special case of the
universal family of hyperplane sections. As usual, we denote by $\TZ$ the (possibly
disconnected) covering space of $B_0$ determined by the local system $\shHZ$, and by
$\eps \colon \TZ \to T(F_{-1} \Mmod)$ the holomorphic mapping induced by the
polarization. Fix some $K \geq 0$. According to the general result in
\theoremref{thm:extension}, we have a finite holomorphic mapping
\[
	\epst \colon \TZt(K) \to T(F_{-1} \Mmod)
\]
from a normal analytic space $\TZt(K)$ that contains $\TZ(K)$ as a dense open subset.
Also recall from \corollaryref{cor:Hdg} that we defined the extended locus of Hodge
classes as the preimage of the zero section. The fact that $F_{-1} \Mmod$ is a quotient of
an ample vector bundle has the following interesting consequence; it was predicted by
Clemens several years ago.

\begin{theorem}
The analytic space $\TZt(K)$ is holomorphically convex; every compact analytic
subset of dimension $\geq 1$ lies inside the extended locus of Hodge classes.  
\end{theorem}

\begin{proof}
For a discussion of holomorphic convexity, see \cite{Cartan}. The result in
\theoremref{thm:ample} shows that $T(F_{-1} \Mmod)$ embeds into the holomorphic
vector bundle
\[
	E = T \Bigl( H^0 \bigl( X, \OmX^{2n+1} \tensor L^{n} \bigr) \tensor \shO_B(n) \Bigr).
\]
Since $E$ is the dual of an ample vector bundle, the zero section can be
contracted to produce a Stein space $Y$; in particular, $E$ is holomorphically
convex. Because $\epst \colon \TZt(K) \to T(F_{-1} \Mmod)$ is finite, it follows that
$\TZt(K)$ is proper over $Y$, and therefore holomorphically convex. Every compact
analytic subset of positive dimension has to map into the zero section of $E$, and
must therefore be contained in the extended locus of Hodge classes.
\end{proof}

\providecommand{\bysame}{\leavevmode\hbox to3em{\hrulefill}\thinspace}
\providecommand{\MR}{\relax\ifhmode\unskip\space\fi MR }
\providecommand{\MRhref}[2]{%
  \href{http://www.ams.org/mathscinet-getitem?mr=#1}{#2}
}
\providecommand{\href}[2]{#2}

\end{document}